\documentclass[reqno]{amsart}

\usepackage[a4paper,left=1in, right=1in, top=1in, bottom=1in]{geometry}
\usepackage[english]{babel}
\usepackage{tikz}
\usepackage{mathrsfs}
\usepackage{amsfonts,amsmath,amsthm,amssymb}
\usepackage{graphicx}
\usepackage{enumerate} 
\usepackage{subfigure}
\usepackage{float}

\usepackage[linktocpage]{hyperref}
\hypersetup{colorlinks=true,linkcolor=blue!85,citecolor=blue!85,urlcolor=blue!85}

\numberwithin{equation}{section}

\newtheorem{theorem}{Theorem}[section]

\newtheorem{lemma}[theorem]{Lemma}

\newcommand{\N}{\mathbb{N}}

\AtBeginDocument{{\noindent\small
This is a preprint version of the paper published open access 
in \emph{Journal of Computational and Applied Mathematics} 
at \url{https://doi.org/10.1016/j.cam.2024.116471}.}
\vspace{9mm}}

\begin{document}
	
\title[Analysis of a Shear beam model with suspenders in thermoelasticity of type III]{%
Analysis of a Shear beam model with suspenders\\ in thermoelasticity of type III}

\author[M. Chabekh, N. Chougui, D. F. M. Torres]{Meriem Chabekh, Nadhir Chougui, Delfim F. M. Torres}

	
\address{Meriem Chabekh\newline
\indent Laboratory of Applied Mathematics, 
Ferhat Abbas S\'etif 1  University, 
S\'etif 19000, Algeria}
\email{meriem.chabekh@univ-setif.dz}

\address{Nadhir Chougui\newline
\indent Laboratory of Applied Mathematics, 
Ferhat Abbas S\'etif 1  University, 
S\'etif 19000, Algeria}
\email{nadhir.chougui@univ-setif.dz}

\address{Delfim F. M. Torres\newline
\indent Center for Research and Development in Mathematics and Applications (CIDMA),\newline 
\indent Department of Mathematics, University of Aveiro, 3810-193 Aveiro, Portugal\newline
\indent and Research Fellow of INTI International University, 71800 Nilai, Negeri Sembilan, Malaysia}
\email{delfim@ua.pt}
	
	

	
	
\begin{abstract}
We conduct an analysis of a one-dimensional linear problem 
that describes the vibrations of a connected suspension bridge. 
In this model, the single-span roadbed is represented as a 
thermoelastic Shear beam without rotary inertia. 
We incorporate thermal dissipation into the transverse displacement 
equation, following Green and Naghdi's theory. Our work demonstrates 
the existence of a global solution by employing classical 
Faedo--Galerkin approximations and three a priori estimates.
Furthermore, we establish exponential stability through the 
application of the energy method. For numerical study, 
we propose a spatial discretization using finite elements 
and a temporal discretization through an implicit Euler scheme. 
In doing so, we prove discrete stability properties and a priori 
error estimates for the discrete problem. To provide a practical 
dimension to our theoretical findings, we present 
a set of numerical simulations.
\end{abstract}	

\maketitle

	
\section{Introduction}

A cable-suspended beam is a structural design comprising 
a beam that is upheld by one or more cables. These cables 
have the dual role of bearing the beam's weight and preserving 
its shape to enhance stability and provide additional support. 
Cable-suspended beams find applications in a range of engineering 
contexts, including cable-stayed bridges 
and suspension bridges \cite{MR4636130}.

In this paper, we address a thermomechanical problem associated 
with a cable-suspended beam structure, exemplified by the suspension bridge. 
The key characteristic of this structure is that the roadbed's sectional 
dimensions are significantly smaller than its length (span of the bridge). 
Suspension bridges with large span lengths exhibit higher flexibility 
in comparison to alternative bridge structures. This increased flexibility 
renders them vulnerable to various dynamic loads, including wind, earthquakes, 
and the movement of vehicles. The distinctive structural features of suspension 
bridges elevate the significance of understanding their dynamic response 
to oscillations, presenting a crucial engineering challenge.
Therefore, we model the roadbed as an extensible thermoelastic Shear-type beam.
The primary suspension cable is represented as an elastic string, 
and it is linked to the roadbed through a distributed network of elastic springs. 
Our model aligns with the configuration of a Shear-suspended-beam system 
within thermoelasticity of type III, as described in the following:
\begin{equation}
\label{main:system*}	
\begin{cases}
\rho u _{tt}-\alpha u_{xx} -\lambda\left( \varphi - u\right)
 +\mu u_t =0, &\text{in}\;  (0,L) \times (0,\infty),\\ 
\rho _{1}\varphi _{tt} -K\left( \varphi _{x}+\psi \right)_{x}
+\lambda\left( \varphi - u\right)+\gamma \varphi_t +\beta \theta_x=0, 
&\text{in}\;  (0,L) \times (0,\infty),\\ 
-b\psi _{xx} +K\left(
\varphi _{x}+\psi \right)=0, &\text{in}\; (0,L) \times (0,\infty),\\ 
\rho_3\theta_{tt}-\delta \theta_{xx} +\beta\varphi_{xtt}
-\kappa\theta_{xxt}=0, &\text{in}\;  (0,L) \times (0,\infty),\\ 
u(0,t)=u(L,t)=\varphi(0,t)=\varphi(L,t)=0, & t\geqslant 0,\\
\psi(0,t)=\psi(L,t)=\theta(0,t)=\theta(L,t)=0, & t\geqslant 0,\\
u(x,0)=u_0(x),\; u_t(x,0)=u_1(x), \;\varphi(x,0)=\varphi_0(x),\; 
\varphi_t(x,0)=\varphi_1(x), & x \in (0,L),\\ 
\psi(x,0)=\psi_0(x),\;\theta(x,0)=\theta_0(x),
\;\theta_t(x,0)=\theta_1(x),& x \in (0,L).
\end{cases}
\end{equation}
Here, the symbol $t$ represents the time variable, 
while $x$ denotes the distance along the centerline 
of the beam in its equilibrium configuration. Both $L$ 
and the beam length coincide with the cable length. 
Within this framework, we use the following notations:
$u$ represents the vertical displacement of the vibrating spring in the primary cable;
$\varphi$ signifies the transverse displacement or vertical deflection 
of the beam's cross-section;
$\psi$ denotes the angle of rotation of a cross-section;
while $\theta$ is employed to represent the thermal moment of the beam.
We consider the suspender cables (ties) as linear elastic springs 
with a shared stiffness parameter $\lambda > 0$, and the constant 
$\alpha > 0$ defines the elastic modulus of the string that connects 
the cable to the deck. The constants $\rho$, $\mu$, $\rho_1$, 
$\rho_3$, $K$, $\gamma$,  $\kappa$, $\delta$, 
and $\beta \neq 0$, employed in the problem formulation,
are positive. The initial data, denoted as 
$(u_0, u_1, \varphi_0, \varphi_1, \psi_0, \theta_0, \theta_1)$, 
belong to an appropriate functional space.

For many years, there has been substantial research interest 
in the dynamic behavior and nonlinear vibrations of suspension bridges
\cite{Rubin1,Rubin2,Ahmed,Matas}. Suspension bridges are complex 
structures with distinctive dynamic characteristics, and they can 
display nonlinear responses to various external forces and loads. 
This nonlinearity may stem from factors such as large-amplitude vibrations, 
material properties, and environmental conditions. 
The appearance of string-beam systems that model 
a nonlinear coupling of a beam (the roadbed) and main cable (the string) 
were born out of the pioneering works of Lazer and McKenna \cite{Lazer,McKenna}.  
Elimination of nonlinear coupling terms in the equations of motion governing 
vertical and torsional vibrations of suspension bridges is achieved by linearization, 
as indicated in reference \cite{Abdel-Ghaffar}. Previous approaches 
have often employed models where the roadbed was based on the Euler-Bernoulli beam theory 
\cite{Bochicchio1,Bochicchio2,Bochicchio3,Bochicchio4}. The Timoshenko beam theory has 
also demonstrated superior performance in anticipating the vibrational response of a beam 
compared to a model grounded in the classical Euler-Bernoulli beam theory \cite{Hayashikawa}.

It is worth noting that the Euler--Bernoulli beam theory and the 
Timoshenko beam theory are two different approaches to model the behavior of beams. 
The Euler--Bernoulli beam theory assumes that the beam is slender and that the cross-section 
remains plane and perpendicular to the longitudinal axis of the beam during deformation. 
This theory is suitable for modeling long and slender beams that are subjected to bending loads. 
On the other hand, the Timoshenko beam theory takes into account the effects of shear deformation 
and rotational bending, which  makes this theory more accurate for analyzing short and thick beams. 
The Timoshenko equations \cite{Timoshenko} improved the Euler--Bernoulli and the 
Rayleigh beam models and such equations are given by
\begin{equation*}
\begin{cases}
\rho _{1}\varphi _{tt} 
-K\left( \varphi _{x}+\psi \right)_{x}=0,\\ 
\rho _{2}\psi _{tt}-b\psi _{xx} +K\left(
\varphi _{x}+\psi \right) =0.\\ 
\end{cases}
\end{equation*}

The Euler--Bernoulli beam theory used to represent 
the single-span roadbed has neglected the effects of shear deformation 
and rotary inertia. These effects can be taken into account by using 
more accurate models, such as the Timoshenko beam theory to have 
a deeper range of applicability and to be physically more realistic 
\cite{Arnold, Labuschagne}. Bochicchio et al. examined 
a linear problem that characterizes 
the vibrations of a coupled suspension bridge \cite{Bochicchio}. 
In their analysis, the single-span roadbed is represented as an 
extensible thermoelastic beam following the Timoshenko model, 
with heat governed by Fourier's law. They demonstrated the exponential 
decay property by employing the established semi-group theory 
and the energy method. Additionally, they conducted several 
numerical experiments to further support their findings.
Mukiawa et al. utilized the same methods to establish the 
existence and uniqueness of a weak global solution and 
to demonstrate exponential stability in 
a thermal-Timoshenko-beam system \cite{Mukiawa}. 
This system incorporated suspenders and Kelvin--Voigt damping 
and was founded on the principles of thermoelasticity, 
as described by Cattaneo's law.

The Timoshenko beam formulation is generally considered for  
determining the dynamic response of beams at higher frequencies (higher wave numbers), 
where there exist two frequency values. As a result, when there are two wave speeds 
$\sqrt{K/\rho_1}$ and $\sqrt{b/\rho_2}$, one of the wave speeds has an infinite speed 
(blow-up) for lower wave numbers. To overcome this physical drawback, simplified beam 
models have been introduced, as the truncated version for dissipative 
Timoshenko type system \cite{Almeida}, and the Shear beam model given by 
\begin{equation*}
\begin{cases}
\rho _{1}\varphi _{tt} -K\left( \varphi _{x}+\psi \right)
_{x}=0,\\ 
-b\psi _{xx} +K\left(
\varphi _{x}+\psi \right) =0,\\ 
\end{cases}
\end{equation*}
which has only one finite wave speed $\sqrt{K/\rho_1}$ for all wave numbers. 
Almeida J\'{u}nior et al. \cite{Ramos1} and Ramos et al. \cite{Ramos2}  
were pioneers in investigating the well-posedness and stability characteristics 
of the Shear beam model. This model constitutes an improvement over the 
Euler--Bernoulli beam model by adding the shear distortion effect but without 
rotary inertia. Consequently, one can analyze long-span bridges, unlike the Timoshenko 
beam theory, which is better suited for modeling beams with relatively short spans. 
This is precisely our focus here -- the examination of the Shear model 
\eqref{main:system*} with suspenders and thermal dissipation, 
introduced by thermoelasticity of Type III.

The second major aspect of problem \eqref{main:system*} under investigation 
in this paper concerns what is referred to as thermoelasticity of type III.
In classical theory of thermoelasticity, heat conduction 
is typically described using Fourier's law for heat flux. 
This theory has the un-physical property
that a sudden temperature change at some point will be felt instantly 
everywhere (that is, the heat propagation has an infinite speed).
However, experiments have shown that the speed of the thermal wave propagation 
in some dielectric crystals at low temperatures is limited. This phenomenon 
in dielectric crystals  is called the \emph{second sound}. 
In the modern theory of thermal propagation, there are several ways to overcome 
this physical  paradox. The most known is the one that proposes replacing 
Fourier's law of heat flux with Cattaneo's law \cite{Cattaneo}, in order
to obtain a heat conduction equation of hyperbolic type that describes 
the wave nature of heat propagation at low temperatures. At the turn of the century, 
Green and Naghdi introduced three other theories (known as thermoelasticity Type I, Type II, 
and Type III), based on the equality of entropy rather than the usual entropy inequality 
\cite{Green1,Green2,Green3}. In each theory, the heat flux is determined by different 
appropriate assumptions. These three theories give a comprehensive and logical explanation 
that embodies the transmission of a thermal pulse and modifies the occurrence of the 
infinite unphysical speed of heat propagation induced by the classical theory 
of heat conduction. When the theory of type I is linearized, it aligns with the classical 
system of thermoelasticity. The systems arising in thermoelasticity of type III exhibit 
dissipative characteristics, while those in type II do not sustain energy dissipation. 
It is a limiting case of thermoelasticity type III 
\cite{Chandrasekharaiah1,Chandrasekharaiah2,Messaoudi,Quintanilla,Zhang}.

Numerically, the finite element method has been utilized in various 
research studies related to control systems \cite{Bernardi,ElArwadi}.

The paper's structure is outlined as follows. In Section~\ref{sec:02}, 
we provide some preliminary information. To demonstrate the global 
existence and uniqueness of solutions, in Section~\ref{sec:03} 
we revisit the Faedo--Galerkin method, 
along with three estimates, previously discussed in references 
such as \cite{Djilali,Lions}. In Section~\ref{sec:04}, we employ 
the energy method to construct several Lyapunov functionals, 
establishing the property of exponential decay.
In Section~\ref{sec:05}, we propose a finite element 
discretization approach to solve the problem at hand. We obtain 
discrete stability results and a priori error estimates. Finally, 
in Section~\ref{sec:06}, we present some numerical simulations using MATLAB. 

Throughout the paper, the symbols $C$ and $C_i$ 
are used to represent various positive constants.


\section{Preliminaries}
\label{sec:02}

In the following, $(\cdot,\cdot)$ represents the scalar product of 
$L_2(0,L)$, and $\lVert \cdot \rVert$ denotes the norm 
$\lVert \cdot \rVert_{L_2(0,L)}$. In order to exhibit the dissipative  
nature of system (\ref{main:system*}), we introduce the new variable
\begin{equation}
\label{new variable}	
w(x,t)=\int_{0}^{t}\theta(x,s)ds +\eta(x),
\end{equation}
where $\eta(x)$ solves
\begin{equation}
\label{solution}	
\begin{cases}
\delta\Delta\eta
=\rho_3\theta_1-\kappa\Delta\theta_0+\beta \nabla \cdot \varphi_1,\\
\eta(0)=\eta(1)=0.
\end{cases}
\end{equation}
Performing an integration  of ${\text{(\ref{main:system*})}}_\text{4}$ 
with respect to $t$ and taking into account (\ref{new variable}), we get
\begin{equation*}
\begin{split}
\int_{0}^{t}\left(\rho_3 \theta_{tt} -\delta \theta_{xx}
+\beta \varphi_{xtt}-\kappa\theta_{xxt}\right)  ds&=0\\
\rho_3w_{tt}-\rho_3\theta_1-\left(\delta w_{xx}-\delta\Delta\eta \right)
+\beta\varphi_{xt}-\beta\nabla \cdot \varphi_1
-\kappa w_{xxt} +\kappa \Delta \theta_0&=0.
\end{split}
\end{equation*} 
Then (\ref{solution}) gives
\begin{equation*}
\rho_3 w_{tt}-\delta w_{xx}+\beta\varphi_{xt}-\kappa w_{xxt}=0,
\end{equation*}
and problem (\ref{main:system*}) takes the form
\begin{equation}
\label{main:system}	
\begin{cases}
\rho u _{tt}-\alpha u_{xx} -\lambda\left( \varphi - u\right)
+\mu u_t =0, &\text{in}\;  (0,L) \times (0,\infty),\\ 
\rho _{1}\varphi _{tt} -K\left( \varphi _{x}+\psi \right)_{x}
+\lambda\left( \varphi - u\right)+\gamma \varphi_t +\beta w_{xt}=0, 
&\text{in}\;  (0,L) \times (0,\infty),\\ 
-b\psi _{xx} +K\left(
\varphi _{x}+\psi \right)=0, &\text{in}\; (0,L) \times (0,\infty),\\ 
\rho_3 w_{tt}-\delta w_{xx} +\beta\varphi_{xt}-\kappa w_{xxt}=0, 
&\text{in}\;  (0,L) \times (0,\infty),\\ 
u(0,t)=u(L,t)=\varphi(0,t)=\varphi(L,t)=0, & t\geqslant 0,\\
\psi(0,t)=\psi(L,t)=w(0,t)=w(L,t)=0, & t\geqslant 0,\\
u(x,0)=u_0(x),\; u_t(x,0)=u_1(x), \;\varphi(x,0)=\varphi_0(x),\; 
\varphi_t(x,0)=\varphi_1(x), & x \in (0,L),\\ 
\psi(x,0)=\psi_0(x),\;w(x,0)=w_0(x),\;w_t(x,0)=w_1(x),
& x \in (0,L).
\end{cases}
\end{equation}

\begin{lemma}
Let $(u,\varphi,\psi,w)$ be the solution of (\ref{main:system}). 
Then the energy functional $E$, defined by
\begin{equation}
\label{Energy funct}	
E(t)=\dfrac{1}{2}\int_{0}^{L}\left( \rho u_t^2 +\alpha u_x^2
+\lambda(\varphi-u)^2+ \rho_1\varphi_t^2 
+K{(\varphi_x+\psi)}^2 +b\psi_x^2 +\rho_3 w_t^2+\delta w_x^2\right)dx, 
\end{equation}
satisfies
\begin{equation}
\label{dissipation eq}
\dfrac{dE(t)}{dt}=-\mu\int_{0}^{L}u_t^2 dx-\gamma\int_{0}^{L}
\varphi_t^2 dx-\kappa \int_{0}^{L}w_{xt}^2 dx
\leqslant 0,\; \forall\; t\geqslant 0.
\end{equation}
\end{lemma}

\begin{proof}
We multiply the  equations in (\ref{main:system}) 
by $u_t$, $\varphi_t$, $\psi_t$, $w_t$, respectively,  
and integrate by parts to get
\begin{equation}
\label{eq energy1}	
\dfrac{1}{2}\dfrac{d}{dt}\left( \rho \int_{0}^{L}u_t^2dx\right)  
+\dfrac{1}{2}\dfrac{d}{dt}\left( \alpha \int_{0}^{L}u_x^2 dx\right)
+\dfrac{1}{2}\dfrac{d}{dt}\left(  \lambda \int_{0}^{L}u^2 dx\right)
-\lambda \int_{0}^{L}\varphi u_t dx +\mu \int_{0}^{L} u_t^2 dx=0,
\end{equation}
\begin{equation}
\begin{split}
\dfrac{1}{2}\dfrac{d}{dt}\left( \rho_1 \int_{0}^{L}\varphi_t^2dx\right)  
&+\dfrac{1}{2}\dfrac{d}{dt}\left( K \int_{0}^{L}\varphi_x^2 dx\right) 
+ K \int_{0}^{L}\psi \varphi_{xt} dx
+\dfrac{1}{2}\dfrac{d}{dt}\left(  \lambda \int_{0}^{L}\varphi^2 dx\right)\\
&-\lambda \int_{0}^{L}u\varphi_t  dx
+\gamma \int_{0}^{L} \varphi_t^2 dx
+\beta \int_{0}^{L} w_{xt}\varphi_t dx=0,
\end{split}
\end{equation}
\begin{equation}
\dfrac{1}{2}\dfrac{d}{dt}\left(b\int_{0}^{L}\psi^2_{x}dx\right) 
+ \dfrac{1}{2}\dfrac{d}{dt}\left(K \int_{0}^{L} \psi^2 dx \right) 
+ K\int_{0}^{L}\varphi_x \psi_t dx=0,
\end{equation}
\begin{equation}
\label{eq energy2}	
\dfrac{1}{2}\dfrac{d}{dt}\left(\rho_3\int_{0}^{L}w_t^2dx \right) 
+\kappa  \int_{0}^{L} w_{xt}^2  dx  + \beta \int_{0}^{L}
\varphi_{xt} w_t dx+\dfrac{1}{2}\dfrac{d}{dt}
\left(\delta\int_{0}^{L}w_x^2dx \right)=0.
\end{equation}
Adding  (\ref{eq energy1})--(\ref{eq energy2}), 
we obtain (\ref{dissipation eq}). These calculations 
are done for any regular solution. Nonetheless, 
the same result holds valid for weak solutions 
by using density arguments.
\end{proof}
 
\begin{theorem}[Aubin--Lions--Simon theorem, see p.~102 of \cite{Boyer}]
\label{theorem:2.1} 
Let $B_0\subset B_1 \subset B_2$ be three Banach spaces. 
We assume that the embedding of $B_1$ in $B_2$ is continuous and that
the embedding of $B_0$ in $B_1$ is compact. Let $p$ and $r$  such that 
$1\leqslant p, r \leqslant +\infty$. For $T>0$, we define
\begin{equation*}
W_{p,r}=\left\lbrace \chi \in L^p(]0,T[,B_0), 
\dfrac{d\chi}{dt} \in L^r(]0,T[,B_2) \right\rbrace.
\end{equation*}
\begin{enumerate}[i)]
\item If $p< +\infty$, the embedding of $W_{p,r}$ 
in $L^p(]0,T[,B_1)$ is compact.
\item If $p= +\infty$ and if $r>1$, 
the embedding of $W_{p,r}$ in $C([0,T],B_1)$ is compact.
\end{enumerate}
\end{theorem}
 

\section{Global well‑posedness}
\label{sec:03}

In this section we prove the existence and uniqueness 
of regular and weak solutions
for system (\ref{main:system}) by using 
the Faedo--Galerkin method. 

\begin{theorem} 
\label{main results}
\ 
\begin{enumerate} [(i)]
\item If the initial data $u_0,\varphi_0,\psi_0,w_0$ $\in H^1_0(0,L)$ 
and $u_1,\varphi_1, w_1\in  L^2 (0,L)$, then problem (\ref{main:system}) 
has a unique weak solution satisfying
\begin{equation}
\label{weak:sol}		
u,\varphi,w  \in C([0,T], H^1_0 (0,L))
\cap C^1([0,T],L^2 (0,L)),\; \psi \in C([0,T], H^1_0 (0,L)). 
\end{equation}
\item If the initial data $u_0, \varphi_0,\psi_0,w_0$ 
$\in H^2(0,L)\cap H^1_0 (0,L)$ and  
$u_1,\varphi_1, w_1\in  H^1_0 (0,L)$, then 
the weak solution \eqref{weak:sol} has higher regularity:
\begin{equation}
\begin{split}
u,\varphi, w  \in C([0,T],H^2(0&,L)\cap H^1_0 (0,L))
\cap C^1([0,T], H^1_0 (0,L))\cap C^2([0,T], L^2 (0,L)),\\
&\psi  \in 	C([0,T],H^2(0,L)\cap H^1_0 (0,L)).
\end{split}
\end{equation}
\end{enumerate}
\end{theorem}

\begin{proof}
We prove the result in six steps.

Step 1: approximated problem. Suppose that $u_0$, $\varphi_0$, 
$\psi_0$, $w_0$ $\in H^2(0,L)\cap H^1_0(0,L)$, $u_1$,  
$\varphi_1$, $w_1$ $\in H^1_0(0,L)$. Let $\left\lbrace v_i\right\rbrace_{i\in \N}$ 
be a smooth orthonormal basis of $L^2(0,L)$, which is also orthogonal in 
$ H^2(0,L)\cap H^1_0(0,L)$, given by the eigenfunctions of 
$-\Delta v=\varsigma v$ with boundary condition $v(0)=v(L)=0$, such that
\begin{equation}
-\Delta v_i= \varsigma_i v_i.
\end{equation}
Here $\varsigma_i$ is the eigenvalue corresponding to $v_i$. 
For any $m \in \N$, denote by $V_m$ the finite-dimensional subspace
\begin{equation*}
\begin{split}
V_m&=span\left\lbrace v_1, v_2, \ldots, v_m\right\rbrace 
\subset H^2(0,L)\cap H^1_0(0,L).
\end{split} 
\end{equation*}
To construct the Galerkin approximation $(u^m, \varphi^m,\psi^m, w^m)$ 
of the solution, let us denote
\begin{equation*}
u^m(t)=\sum_{i=1}^{m}g_{im}(t)v_i(x),\; \varphi^m(t)
=\sum_{i=1}^{m}\hat{g}_{im}(t)v_i(x),
\end{equation*}
\begin{equation*}
\psi^m(t)=\sum_{i=1}^{m} f_{im}(t)v_i(x),\; w^m(t)
=\sum_{i=1}^{m} \hat{f}_{im}(t)v_i(x),
\end{equation*}
where functions $g_{im}$, $\hat{g}_{im}$, $f_{im}$, 
$\hat{f}_{im}\in C^2[0,T]$ are given by the solution 
of the approximated system
\begin{equation}
\label{appro:problem}	
\begin{cases}
\rho (u _{tt}^m( t),v_i) +\alpha(u_{x}^m (t) ,v_{ix})
-\lambda(\varphi^m(t)-u^m(t),v_i)
+\mu (u^m _{t}\left(t\right),v_i)=0,\\ 
\begin{split}
\rho _{1}(\varphi _{tt}^m( t),v_i) +K( \varphi _{x}^m (t)
&+\psi^m (t) ,v_{ix})+\lambda (\varphi^m(t)-u^m(t),v_i)\\
&+\gamma (\varphi^m _{t}\left(t\right),v_i)+\beta(w_{xt}^m(t),v_i)=0,
\end{split}\\
b(\psi^m _{x}\left(t\right),v_{ix}) +K\left(
\varphi^m _{x}(t)+\psi^m(t),v_i \right)=0, \\ 
\rho_3 (w^m_{tt}(t),v_i)+\delta(w^m_x(t),v_{ix})
+\beta(\varphi_{xt}^m(t), v_i)+\kappa(w^m_{xt}(t),v_{ix})=0,
\end{cases}
\end{equation}
with the initial conditions 
\begin{equation}
u^m(0)=u^m_0= \sum_{i=1}^{m}(u_0 ,v_i)v_i 
\xrightarrow{m \rightarrow\infty} u_0
\;\; \text{in}\; H^2(0,L)\cap H^1_0(0,L),
\end{equation}
\begin{equation}
\varphi^m(0)=\varphi^m_0= \sum_{i=1}^{m}(\varphi_0 ,v_i)
v_i \xrightarrow{m \rightarrow\infty} \varphi_0
\;\; \text{in}\; H^2(0,L)\cap H^1_0(0,L),
\end{equation}
\begin{equation}
\psi^m(0)=\psi^m_0= \sum_{i=1}^{m}(\psi_0 ,v_i)v_i 
\xrightarrow{m \rightarrow\infty} 
\psi_0\;\; \text{in}\; H^2(0,L)\cap H^1_0(0,L),
\end{equation}
\begin{equation}
w^m(0)=w^m_0= \sum_{i=1}^{m}(w_0 ,v_i)v_i 
\xrightarrow{m \rightarrow\infty} w_0\;\; 
\text{in}\; H^2(0,L)\cap H^1_0(0,L),
\end{equation}
\begin{equation}
u^m_t(0)=u^m_1= \sum_{i=1}^{m}(u_1 ,v_i)v_i 
\xrightarrow{m \rightarrow\infty} u_1\;\; 
\text{in}\;  H^1_0(0,L),
\end{equation}
\begin{equation}
\varphi^m_t(0)=\varphi^m_1= \sum_{i=1}^{m}(\varphi_1 ,v_i)v_i 
\xrightarrow{m \rightarrow\infty} \varphi_1\;\; 
\text{in}\;  H^1_0(0,L),
\end{equation}
\begin{equation}
\label{initial cond}	
w^m_t(0)=w^m_1= \sum_{i=1}^{m}(w_1 ,v_i)v_i 
\xrightarrow{m \rightarrow\infty} w_1\;\; 
\text{in}\;  H^1_0(0,L). 
\end{equation}
From the theory of ODEs, system (\ref{appro:problem})--(\ref{initial cond}) 
has a unique local solution $(u^m(t),\varphi^m(t), \psi^m(t), w^m(t))$
defined on $[0, T_m)$ with $0<T_m<T$. Our next step is to show that the 
local solution is extended to $[0,T]$ for any given $T>0$.

Step 2: first a priori estimate. Multiplying 
${\text{(\ref{appro:problem})}}_\text{1}$, 
${\text{(\ref{appro:problem})}}_\text{2}$, 
${\text{(\ref{appro:problem})}}_\text{3}$ 
and ${\text{(\ref{appro:problem})}}_\text{4}$ 
by $g'_{im}$,  $\hat{g}'_{im}$, $f'_{im}$ and $\hat{f}'_{im}$, respectively,
summing up over $i$ from $1$ to $m$,  taking the sum 
of the resulting equations and using
integration by parts, we obtain that
\begin{equation}
\label{estimate1 eq1}	
\begin{split}
&\dfrac{d}{dt} \left(\dfrac{\rho}{2} {\rVert u^m_t(t)\lVert}^2
+\dfrac{\alpha}{2} {\rVert u^m_x(t)\lVert}^2 +\dfrac{\lambda}{2} 
{\rVert \varphi^m(t)- u^m(t)\lVert}^2
+ \dfrac{\rho_1}{2} {\rVert\varphi^m_t(t)\lVert}^2  \right) \\
&+\dfrac{d}{dt}\left(\dfrac{K}{2} {\rVert\varphi^m_x(t)
+\psi^m(t)\lVert}^2+\dfrac{b}{2}  {\lVert\psi^m_x(t)\rVert}^2
+\dfrac{\rho_3}{2} {\lVert w_t^m(t)\rVert}^2+\dfrac{\delta}{2}  
{\lVert w_x^m(t)\rVert}^2 \right)\\
&+ \mu {\lVert u_t^m(t)\rVert}^2
+\gamma {\lVert \varphi_t^m(t)\rVert}^2
+\kappa {\lVert w_{xt}^m(t)\rVert}^2=0.
\end{split}
\end{equation}
Integrating (\ref{estimate1 eq1}) over $(0,t)$, we find
\begin{equation}
\label{estimate1:eq2}	
\mathcal{E}^m_1(t) +\mu\int_{0}^{t} {\lVert u_t^m(s)\rVert}^2 ds
+\gamma \int_{0}^{t}{\lVert \varphi_t^m(s)\rVert}^2ds
+\kappa \int_{0}^{t}{\lVert w_{xt}^m(s)\rVert}^2 ds
=\mathcal{E}^m_1(0),
\end{equation}
where 
\begin{equation}
\label{estimate1:eq3}	
\begin{split}
\mathcal{E}^m_1(t)=&\dfrac{1}{2}\left(\rho{\rVert u^m_t(t)\lVert}^2
+\alpha{\rVert u^m_x(t)\lVert}^2+\lambda {\rVert \varphi^m(t)
-u^m(t)\lVert}^2 +\rho_1{\rVert\varphi^m_t(t)\lVert}^2 \right)\\
&+\dfrac{1}{2}\left( K {\rVert\varphi_x^m(t)+\psi^m(t)\lVert}^2
+b{\rVert \psi^m_x(t)\lVert }^2+ \rho_3 {\lVert w_t^m(t)\rVert}^2
+\delta{\lVert w_x^m(t)\rVert}^2\right).
\end{split} 
\end{equation}
From (\ref{estimate1:eq2}), we deduce that there exists 
$C$, independent of $m$, such that 
\begin{equation}
\label{estimate1:eq4}	
\mathcal{E}^m_1(t)\leqslant \mathcal{E}^m_1(0)
\leqslant C, \; t\geqslant 0.
\end{equation}
Consequently, we have
\begin{equation}
\label{estimate1:eq5}	
\begin{split}
&{\rVert u^m_t(t)\lVert}^2+{\rVert u^m_x(t)\lVert}^2
+ {\rVert \varphi^m(t)-u^m(t)\lVert}^2+ {\rVert\varphi^m_t(t)\lVert}^2\\
&+ {\rVert\varphi_x^m(t)+\psi^m(t)\lVert}^2+{\rVert \psi^m_x(t)\lVert }^2
+ {\lVert w_t^m(t)\rVert}^2+{\lVert w_x^m(t)\rVert}^2\leqslant C.
\end{split} 
\end{equation}
The estimate \eqref{estimate1:eq5} implies 
that the solution $(u^m,\varphi^m,\psi^m,w^m)$ 
exists globally in $[0,\infty)$ and, for any $m\in \N$,
\begin{equation}
\label{estimate1:eq6}	
u^m,\;\varphi^m,\; \psi^m ,\; w^m\;\; 
\text{are bounded in }\; L^{\infty} (0,T; H^{1}_0 (0,L)), 
\end{equation}
\begin{equation}
\label{estimate1:eq7}	
u^m_t,\;\varphi^m_t,\;w^m_t\;\; \text{are bounded in }\; 
L^{\infty} (0,T; L^{2} (0,L)). 
\end{equation}

Step 3: second a priori estimate. Taking the derivative of the 
approximate equations in (\ref{appro:problem}) with respect to $t$, we get 
\begin{equation}
\label{appro:problem2}	
\begin{cases}
\rho(u _{ttt}^m( t),v_i)+\alpha(u^m_{xt}(t),v_{ix}) 
-\lambda( \varphi _{t}^m (t)-u^m_t (t) ,v_{i})
+\mu (u^m _{tt}\left(t\right),v_i) =0,\\
\begin{split}
\rho _{1}(\varphi _{ttt}^m( t),v_i) 
+K( \varphi _{xt}^m (t)&+\psi^m_t (t) ,v_{ix})
+\lambda( \varphi _{t}^m (t)-u^m_t (t) ,v_{i})\\
&+\gamma (\varphi^m _{tt}\left(t\right),v_i)+\beta(w^m_{xtt},v_i) =0,
\end{split}\\ 
b(\psi^m _{xt}\left(t\right),v_{ix}) +K\left(
\varphi^m _{xt}(t)+\psi^m_t(t),v_i \right)=0, \\ 
\rho_3(w_{ttt}^m(t),v_i)+\delta(w_{xt}^m(t),v_{ix})
+\beta(\varphi_{xtt}^m(t),v_i)+\kappa(w^m_{xtt},v_{ix})=0.
\end{cases}
\end{equation}
Multiplying  ${\text{(\ref{appro:problem2})}}_\text{1}$, 
${\text{(\ref{appro:problem2})}}_\text{2}$, 
${\text{(\ref{appro:problem2})}}_\text{3}$ 
and ${\text{(\ref{appro:problem2})}}_\text{4}$  
by $g''_{im}$,  $\hat{g}''_{im}$, $ f''_{im}$ 
and $ \hat{f}''_{im}$, respectively, 
and summing up  over $i$ from $1$ to $m$, it follows that 
\begin{equation}
\label{estimate2:eq1}	
\dfrac{1}{2}\dfrac{d}{dt} \left( \rho {\rVert u^m_{tt}(t)\lVert}^2\right)
+\dfrac{1}{2}\dfrac{d}{dt} \left( \alpha {\rVert u^m_{tx}(t)\lVert}^2\right)
-\lambda\int_{0}^{L}(\varphi^m_t(t)-u^m_t(t))u^m_{tt}(t) dx
+\mu{\rVert u^m_{tt}(t)\lVert}^2=0,
\end{equation} 
\begin{equation}
\begin{split}
\dfrac{1}{2}\dfrac{d}{dt}
& \left( \rho_1 {\rVert\varphi^m_{tt}(t)\lVert}^2\right) 
+K \int_{0}^{L}(\varphi^m_{xt}(t)+\psi^m_t(t) )\varphi^m_{xtt}(t) dx 
+\gamma{\rVert\varphi^m_{tt}(t)\lVert}^2 \\
&+\lambda\int_{0}^{L}(\varphi^m_t(t)-u^m_t(t))\varphi^m_{tt}(t) dx
+\beta\int_{0}^{L}w^m_{xtt}(t)\varphi^m_{tt}(t)dx=0,
\end{split}
\end{equation}
\begin{equation}
\begin{split}
\dfrac{1}{2}\dfrac{d}{dt} \left(b {\lVert\psi^m_{xt}(t)\rVert}^2\right)
+K \int_{0}^{L}(\varphi^m_{xt}(t)+\psi^m_t(t) )\psi^m_{tt}(t) dx=0,
\end{split}
\end{equation}
\begin{equation}
\label{estimate2:eq2}	
\dfrac{1}{2}\dfrac{d}{dt} \left(\rho_3 {\rVert w^m_{tt}(t)\lVert}^2\right)
+\dfrac{1}{2}\dfrac{d}{dt} \left(\delta {\rVert w^m_{xt}(t)\lVert}^2\right)
+\beta \int_{0}^{L}\varphi^m_{xt}(t)w_{tt}^m(t) dx
+\kappa{\rVert w^m_{xtt}(t)\lVert}^2=0.
\end{equation}
Adding up (\ref{estimate2:eq1})--(\ref{estimate2:eq2}) and using
integration by parts, we obtain that
\begin{equation}
\label{estimate2:the:sum}	
\begin{split}
&\dfrac{d}{dt}\left(\dfrac{\rho}{2} {\rVert u^m_{tt}(t)\lVert}^2
+\dfrac{\alpha}{2} {\rVert u^m_{xt}(t)\lVert}^2 
+\dfrac{\lambda}{2} {\rVert \varphi_t^m(t)- u_t^m(t)\lVert}^2
+ \dfrac{\rho_1}{2} {\rVert\varphi^m_{tt}(t)\lVert}^2  \right) \\
&+\dfrac{d}{dt}\left(\dfrac{K}{2} {\rVert\varphi^m_{xt}(t)
+\psi_t^m(t)\lVert}^2+\dfrac{b}{2}  {\lVert\psi^m_{xt}(t)\rVert}^2
+\dfrac{\rho_3}{2} {\lVert w_{tt}^m(t)\rVert}^2
+\dfrac{\delta}{2}{\lVert w_{xt}^m(t)\rVert}^2 \right)\\
&+\mu {\lVert u_{tt}^m(t)\rVert}^2+\gamma {\lVert 
\varphi_{tt}^m(t)\rVert}^2+\kappa {\lVert w_{xtt}^m(t)\rVert}^2=0.
\end{split}
\end{equation}
Now, integrating (\ref{estimate2:the:sum}) over $(0,t)$, yields
\begin{equation}
\mathcal{E}^m_2(t) +\mu \int_{0}^{t}{\lVert u_{tt}^m(s)\rVert}^2ds
+\gamma\int_{0}^{t} {\lVert \varphi_{tt}^m(s)\rVert}^2ds
+\kappa \int_{0}^{t}{\lVert w_{xtt}^m(s)\rVert}^2ds
=\mathcal{E}^m_2(0),
\end{equation}
where 
\begin{equation}
\begin{split}
\mathcal{E}^m_2(t)
=&\dfrac{1}{2}\left(\rho{\rVert u^m_{tt}(t)\lVert}^2
+\alpha{\rVert u^m_{xt}(t)\lVert}^2+\lambda {\rVert \varphi_t^m(t)
-u_t^m(t)\lVert}^2 +\rho_1{\rVert\varphi^m_{tt}(t)\lVert}^2 \right)\\
&+\dfrac{1}{2}\left( K {\rVert\varphi_{xt}^m(t)+\psi_t^m(t)\lVert}^2
+b{\rVert \psi^m_{xt}(t)\lVert }^2+ \rho_3 {\lVert w_{tt}^m(t)\rVert}^2
+\delta{\lVert w_{xt}^m(t)\rVert}^2\right).
\end{split} 
\end{equation}
Thus, there exists $C$ independent of $m$ such that
\begin{equation}
\mathcal{E}^m_2(t)\leqslant \mathcal{E}^m_2(0)
\leqslant C, \; t\geqslant 0,
\end{equation}
accordingly
\begin{equation}
\begin{split}
&{\rVert u^m_{tt}(t)\lVert}^2+{\rVert u^m_{xt}(t)\lVert}^2
+ {\rVert \varphi_t^m(t)-u_t^m(t)\lVert}^2+ {\rVert\varphi^m_{tt}(t)\lVert}^2\\
&+ {\rVert\varphi_{xt}^m(t)+\psi_t^m(t)\lVert}^2+{\rVert \psi^m_{xt}(t)\lVert }^2
+ {\lVert w_{tt}^m(t)\rVert}^2+{\lVert w_{xt}^m(t)\rVert}^2
\leqslant C,
\end{split} 
\end{equation}
and, for any $m\in \N$, we have
\begin{equation}
\label{estimate2:eq4}	
u_t^m,\;\varphi^m_t,\; \psi^m_t,\; w^m_t 
\;\; \text{bounded in }\; 
L^{\infty}(0,T; H^{1}_0 (0,L)), 
\end{equation}
\begin{equation}
\label{estimate2:eq5}	
u^m_{tt},\;\varphi^m_{tt},\; w^m_{tt}
\;\; \text{bounded in }\; 
L^{\infty} (0,T; L^{2} (0,L)). 
\end{equation}

Step 4: third a priori estimate. Replacing $v_i$ by $-v_{ixx}$ 
in (\ref{appro:problem}) and multiplying the resulting equations 
by $g'_{im}$, $\hat{g}'_{im}$, $f'_{im}$ and $\hat{f}'_{im}$, 
respectively, summing over $i$ from $1$ to $m$ and using 
integration by parts, we obtain that
\begin{equation}
\label{estimate3:eq1}	
\dfrac{1}{2}\dfrac{d}{dt} \left( \rho {\rVert u^m_{xt}(t)\lVert}^2\right)
+\dfrac{1}{2}\dfrac{d}{dt} \left( \alpha {\rVert u^m_{xx}(t)\lVert}^2\right)
-\lambda\int_{0}^{L}(\varphi^m_x(t)-u^m_x(t))u^m_{xt}(t) dx
+\mu{\rVert u^m_{xt}(t)\lVert}^2=0,
\end{equation}
\begin{equation}
\begin{split}
\dfrac{1}{2}\dfrac{d}{dt}
& \left( \rho_1 {\rVert\varphi^m_{xt}(t)\lVert}^2\right) 
+K \int_{0}^{L}(\varphi^m_{xx}(t)+\psi^m_x(t) )\varphi^m_{xxt}(t) dx 
+\gamma{\rVert\varphi^m_{xt}(t)\lVert}^2 \\
&+\lambda\int_{0}^{L}(\varphi^m_x(t)-u^m_x(t))\varphi^m_{xt}(t) dx
-\beta\int_{0}^{L}w^m_{xt}(t)\varphi^m_{xxt}(t)dx=0,
\end{split}
\end{equation}
\begin{equation}
\begin{split}
\dfrac{1}{2}\dfrac{d}{dt} \left(b {\lVert\psi^m_{xx}(t)\rVert}^2\right)
-K \int_{0}^{L}(\varphi^m_{x}(t)+\psi^m(t) )\psi^m_{xxt}(t) dx=0,
\end{split}
\end{equation}
\begin{equation}
\label{estimate3:eq2}	
\dfrac{1}{2}\dfrac{d}{dt} \left(\rho_3 {\rVert w^m_{xt}(t)\lVert}^2\right)
+\dfrac{1}{2}\dfrac{d}{dt} \left(\delta {\rVert w^m_{xx}(t)\lVert}^2\right)
+\beta \int_{0}^{L}\varphi^m_{xxt}(t)w_{xt}^m(t) dx
+\kappa{\rVert w^m_{xxt}(t)\lVert}^2=0.
\end{equation}
Taking the sum of (\ref{estimate3:eq1})--(\ref{estimate3:eq2}) 
and integrating over $(0,t)$, one has
\begin{equation}
\begin{split}
\mathcal{E}^m_3(t)+\mu \int_{0}^{t}{\lVert u_{xt}^m(s)\rVert}^2ds
+\gamma\int_{0}^{t} {\lVert \varphi_{xt}^m(s)\rVert}^2ds
+\kappa \int_{0}^{t}{\lVert w_{xxt}^m(s)\rVert}^2ds 
= \mathcal{E}^m_3(0),
\end{split}
\end{equation}
where 
\begin{equation}
\begin{split}
\mathcal{E}^m_3(t)
=&\dfrac{1}{2}\left(\rho{\rVert u^m_{xt}(t)\lVert}^2
+\alpha{\rVert u^m_{xx}(t)\lVert}^2+\lambda {\rVert \varphi_x^m(t)
-u_x^m(t)\lVert}^2 +\rho_1{\rVert\varphi^m_{xt}(t)\lVert}^2 \right)\\
&+\dfrac{1}{2}\left( K {\rVert\varphi_{xx}^m(t)+\psi_x^m(t)\lVert}^2
+b{\rVert \psi^m_{xx}(t)\lVert }^2+ \rho_3 {\lVert w_{xt}^m(t)\rVert}^2
+\delta{\lVert w_{xx}^m(t)\rVert}^2\right).
\end{split} 
\end{equation}
Then, there exists $C$, independent of $m$, such that
\begin{equation*}
\mathcal{E}^m_3(t)\leqslant \mathcal{E}^m_3(0)
\leqslant C, \; t\geqslant 0,
\end{equation*}
which entails that
\begin{equation}
\label{estimate3:eq3}	
\begin{split}
&{\rVert u^m_{xt}(t)\lVert}^2+{\rVert u^m_{xx}(t)\lVert}^2
+ {\rVert \varphi_x^m(t)-u_x^m(t)\lVert}^2+ {\rVert\varphi^m_{xt}(t)\lVert}^2\\
&+ {\rVert\varphi_{xx}^m(t)+\psi_x^m(t)\lVert}^2+{\rVert \psi^m_{xx}(t)\lVert}^2
+ {\lVert w_{xt}^m(t)\rVert}^2+{\lVert w_{xx}^m(t)\rVert}^2\leqslant C.
\end{split} 
\end{equation}
For all $m\in \N$, (\ref{estimate3:eq3}) implies that 
\begin{equation}
\label{estimate3:eq5}	
u^m,\;\varphi^m,\; \psi^m,\; w^m \;\; \text{are bounded in }
\; L^{\infty} (0,T;H^{2}(0,L) \cap H^{1}_0 (0,L)).
\end{equation}

Step 5: passage to the limit. Thanks to   
(\ref{estimate2:eq4}), (\ref{estimate2:eq5}) and (\ref{estimate3:eq5}),  
and passing to a subsequence, if necessary, we have
\begin{equation}
\label{weak:limit}	
\begin{cases}
u^m\rightharpoonup^\ast u \;\; \text{ in }\; 
L^{\infty} (0,T; H^{2}(0,L) \cap H^{1}_0 (0,L)),\\
u^m_t\rightharpoonup^\ast u_t \;\; \text{ in }\; 
L^{\infty} (0,T; H^{1}_0 (0,L)),\\ u^m_{tt}
\rightharpoonup^\ast u_{tt} \;\; 
\text{ in }\; L^{\infty} (0,T; L^{2} (0,L)),
\end{cases}
\end{equation}
\begin{equation}
\begin{cases}
\varphi^m\rightharpoonup^\ast \varphi \;\; 
\text{ in }\; L^{\infty} (0,T; H^{2}(0,L) \cap H^{1}_0 (0,L)),\\
\varphi^m_t\rightharpoonup^\ast \varphi_t \;\; 
\text{ in }\; L^{\infty} (0,T;  H^{1}_0 (0,L)),\\
\varphi^m_{tt}\rightharpoonup^\ast \varphi_{tt} \;\; 
\text{ in }\; L^{\infty} (0,T; L^{2} (0,L)),
\end{cases}
\end{equation}
\begin{equation}
\begin{cases}
\psi^m\rightharpoonup^\ast \psi \;\; 
\text{ in }\; L^{\infty} (0,T; H^{2}(0,L) \cap H^{1}_0 (0,L)),\\
\psi^m_t\rightharpoonup^\ast \psi_t \;\; 
\text{ in }\; L^{\infty} (0,T;  H^{1}_0 (0,L)),
\end{cases}
\end{equation}
\begin{equation}
\begin{cases}
w^m\rightharpoonup^\ast w \;\; \text{ in }\; 
L^{\infty} (0,T; H^{2}(0,L) \cap H^{1}_0 (0,L)),\\
w^m_t\rightharpoonup^\ast w_t \;\; \text{ in }\; 
L^{\infty} (0,T; H^{1}_0 (0,L)),\\ 
w^m_{tt}\rightharpoonup^\ast w_{tt} \;\; 
\text{ in }\; L^{\infty} (0,T;  L^{2} (0,L)).
\end{cases}
\end{equation}
From the above limits, we conclude that
$(u,u_t,\varphi,\varphi_t,\psi,w,w_t)$  
is a strong weak solution with higher regularity, satisfying
\begin{equation*}
u,\varphi,\psi,w \in L^{\infty} (0,T; H^{2}(0,L) 
\cap H^{1}_0 (0,L)),
\quad u_t,\varphi_t, w_t  
\in L^{\infty} (0,T; H^{1}_0 (0,L)).
\end{equation*}
The embedding $H^{2}(0,L)\cap H^{1}_0(0,L)$ in $ H^{1}_0(0,L)$ is compact. 
Note that if we let $B_1=B_2=H^{1}_0 (0,L)$ and $B_0=H^{2}(0,L)\cap H^{1}_0(0,L)$ 
in Aubin--Lions--Simon Theorem~\ref{theorem:2.1}, then we get that the embedding 
of $W_{\infty,\infty}$ in $C(]0,T[,H^{1}_0 (0,L))$ is compact, where
\begin{equation*}
W_{\infty,\infty}=\left\lbrace u^m :\;  
u^m \in L^{\infty}(]0,T[,H^{2}(0,L)\cap H^{1}_0(0,L)), 
u^m_t \in L^{\infty}(]0,T[,H^{1}_0 (0,L)) \right\rbrace.
\end{equation*}
Now, from (\ref{weak:limit}), we deduce that $u^m$ is bounded 
in $W_{\infty,\infty}$ and, therefore, we can
extract a subsequence $(u^\nu)$ of $(u^m)$ such that
\begin{equation}
\label{step5:eq1}	
u^ \nu \longrightarrow u \;\; \text{ in }\; 
C ([0,T], H^{1}_0 (0,L)).
\end{equation}
Similarly, we obtain 
\begin{equation*}
\psi^ \nu \longrightarrow \psi \;\; \text{ in }\; C ([0,T], H^{1}_0 (0,L)),
\end{equation*}
\begin{equation*}
\varphi^ \nu \longrightarrow \varphi \;\; \text{ in }\; C ([0,T], H^{1}_0 (0,L)),
\end{equation*}
\begin{equation*}
w^ \nu \longrightarrow w \;\; \text{ in }\; C ([0,T], H^{1}_0 (0,L)).
\end{equation*}
Since  the embedding $H^{1}_0 (0,L)\hookrightarrow L^{2}(0,L)$ is compact, 
going back again to the compactness Theorem~\ref{theorem:2.1} 
with $B_0=H^{1}_0 (0,L)$, $B_1=B_2=L^{2}(0,L)$ and $\chi=u^m_t$, gives  
\begin{equation}
\label{step5:eq2}	
u^\nu_t \longrightarrow u_t \;\; 
\text{ in }\; C ([0,T], L^{2} (0,L)),
\end{equation}
where here $(u^\nu_t)$ is a subsequence of $(u^m_t)$. Similarly, we find
\begin{equation*}
\varphi^\nu_t \longrightarrow \varphi_t \;\; 
\text{ in }\; C ([0,T], L^{2} (0,L)),
\end{equation*}
\begin{equation*}
w^\nu_t \longrightarrow w_t \;\; 
\text{ in }\; C ([0,T], L^{2} (0,L)).
\end{equation*}
On the other hand, note that 
\begin{equation*}
-u^\nu_{xx}=\varsigma u^\nu\; \text{ and }\; 
C ([0,T], H^{1}_0 (0,L)) \subset C ([0,T], L^{2} (0,L)).
\end{equation*}
Then, from (\ref{step5:eq1}), we infer that 
\begin{equation*}
u^\nu \longrightarrow u \;\; \text{ in }\; 
C ([0,T], H^{2}(0,L)\cap H^{1}_0 (0,L)).
\end{equation*}
Now, making use of (\ref{step5:eq1}) and (\ref{step5:eq2}), 
respectively, with the dominated convergence theorem 
(differentiating under the integral sign), gives that
\begin{equation*}
\lVert u^\nu_t-u_t  \rVert_{C ([0,T], H^{1}_0 (0,L))}
\xrightarrow{\nu \rightarrow\infty} 0,
\end{equation*}
\begin{equation*}
\lVert u^\nu_{tt}-u_{tt}  \rVert_{C ([0,T], L^{2} (0,L))}
\xrightarrow{\nu \rightarrow\infty} 0,
\end{equation*}
which implies that
\begin{equation*}
u^\nu_t \longrightarrow u_t \;\; 
\text{ in }\; C ([0,T], H^{1}_0 (0,L)),
\end{equation*}
\begin{equation*}
u^\nu_{tt} \longrightarrow u_{tt} \;\; 
\text{ in }\; C ([0,T], L^{2} (0,L)).
\end{equation*}
The same arguments are used to determine the limits of  
$\varphi^\nu,\varphi^\nu_t,\varphi^\nu_{tt},
\psi^\nu,w^\nu,w^\nu_{t}$ and $w^\nu_{tt}$.
With these limits, we can pass to the limit of the terms of the 
approximate equations in (\ref{appro:problem})  
to get a strong weak solution to problem (\ref{main:system}).

Step 6: continuous dependence and uniqueness.
Let $(u,u_t,\varphi,\varphi_t,\psi,w,w_t)$ and 
$(\tilde{u},\tilde{u}_t,\tilde{\varphi},\tilde{\varphi}_t,
\tilde{\psi},\tilde{w},\tilde{w}_t)$ be strong solutions 
of problem (\ref{main:system}). Then,
\begin{equation*}
(U,U_t,\Lambda,\Lambda_t,X,\Theta,\Theta_t)
=(u-\tilde{u},u_t-\tilde{u}_t,\varphi-\tilde{\varphi},
\varphi_t-\tilde{\varphi}_t,\psi-\tilde{\psi},w-\tilde{w},w_t-\tilde{w}_t)
\end{equation*} 
satisfies
\begin{equation}
\label{uniqueness:eq1}	
\rho _{}U_{tt}\left( x,t\right) -\alpha U_{xx}\left( x,t\right) 
-\lambda(\Lambda-U)\left( x,t\right)+\mu U_t \left( x,t\right) =0,
\end{equation}
\begin{equation}
\label{uniqueness:eq2}	
\rho _{1}\Lambda _{tt}\left( x,t\right) 
-K\left( \Lambda _{x}+X \right)_{x}\left( x,t\right) 
+\lambda(\Lambda-U)\left( x,t\right)+\gamma\Lambda_t(x,t)
+\beta\Theta_{xt}(x,t)=0,
\end{equation}
\begin{equation}
\label{uniqueness:eq3}	
-bX _{xx}\left( x,t\right) +K\left(\Lambda_{x}
+X \right) \left( x,t\right)=0, 
\end{equation}
\begin{equation}
\label{uniqueness:eq4}	
\rho_3 \Theta_{tt}(x,t)-\delta \Theta_{xx}(x,t)
+\beta \Lambda_{xt}(x,t)-\kappa \Theta_{xxt}(x,t) =0,
\end{equation}
with the initial data 
\begin{equation*}
\begin{split}
U(0)&=u(0)-\tilde{u}(0),\\
U_t(0)&=u_t(0)-\tilde{u}_t(0),\\
\Lambda(0)&=\varphi(0)-\tilde{\varphi}(0),\\
\Lambda_t(0)&=\varphi_t(0)-\tilde{\varphi}_t(0),\\
X(0)&=\psi(0)-\tilde{\psi}(0),\\
\Theta(0)&=w(0)-\tilde{w}(0),\\
\Theta_t(0)&=w_t(0)-\tilde{w}_t(0).
\end{split}
\end{equation*}
Repeating exactly the same arguments used to obtain
the estimate (\ref{dissipation eq}), we get
\begin{equation}
\label{dissipation:eq2}	
\dfrac{d\tilde{E}(t)}{dt}=-\mu{\rVert U_t\lVert }^2
-\gamma{\rVert \Lambda_t\lVert }^2
-\kappa{\rVert \Theta_{xt}\lVert }^2.
\end{equation}
Integrating (\ref{dissipation:eq2}) 
over $(0,t)$, we conclude that   
\begin{equation*}
\tilde{E}(t)\leqslant C_T\tilde{E}(0),\; t\in [0,T],
\end{equation*}
for some constant $C_T > 0$, which proves 
the continuous dependence of the solutions 
on the initial data. In particular, the stronger 
weak solution is unique.

To demonstrate the existence of weak solutions, 
let us consider the initial data $u_0,\varphi_0,\psi_0,w_0$ $\in H^1_0(0,L)$ 
and $u_1,\varphi_1, w_1\in  L^2 (0,L)$ in the approximate 
problem (\ref{appro:problem}). Then, by density, we have 
\begin{equation}
\label{initial_data_week_solution}	
u^m_0,\varphi^m_0,\psi^m_0,w^m_0\rightarrow u_0,\varphi_0,\psi_0,w_0\;\; 
\text{in}\;  H^1_0(0,L),\; u^m_1,\varphi^m_1,w^m_1
\rightarrow u_1,\varphi_1,w_1\;\; \text{in}\;  L^2(0,L).
\end{equation}
Repeating the same steps used in the first estimate, following the same procedure 
already used in the uniqueness of strong solutions for $(U^m,U_t^m,\Lambda^m,
\Lambda_t^m,X^m,\Theta^m,\Theta_t^m)=(u^m-\tilde{u}^m,u_t^m-\tilde{u}_t^m,\varphi^m-\tilde{\varphi}^m,
\varphi_t^m-\tilde{\varphi}_t^m,\psi^m-\tilde{\psi}^m,w^m-\tilde{w}^m,w_t^m-\tilde{w}_t^m)$ 
and taking into account the convergences (\ref{initial_data_week_solution}), 
we deduce that there exists $u$, $\varphi$, $\psi$, and $w$ such that
\begin{equation*}
u^m \longrightarrow u \;\; \text{ in }\; 
C ([0,T], H^{1}_0 (0,L)),
\end{equation*}
\begin{equation*}
\psi^m \longrightarrow \psi \;\; \text{ in }\; C ([0,T], H^{1}_0 (0,L)),
\end{equation*}
\begin{equation*}
\varphi^m \longrightarrow \varphi \;\; \text{ in }\; C ([0,T], H^{1}_0 (0,L)),
\end{equation*}
\begin{equation*}
w^m \longrightarrow w \;\; \text{ in }\; C ([0,T], H^{1}_0 (0,L)),
\end{equation*}
\begin{equation*}	
u^m_t \longrightarrow u_t \;\; 
\text{ in }\; C ([0,T], L^{2} (0,L)),
\end{equation*}
\begin{equation*}
\varphi^m_t \longrightarrow \varphi_t \;\; 
\text{ in }\; C ([0,T], L^{2} (0,L)),
\end{equation*}
\begin{equation*}
w^m_t \longrightarrow w_t \;\; 
\text{ in }\; C ([0,T], L^{2} (0,L)).
\end{equation*}
The uniqueness is obtained making use of the well-known regularization 
procedure, as presented in \cite[Chapter 3, Section 8.2.2]{Lions}.
\end{proof} 


\section{Exponential stability}
\label{sec:04}

The main result of this section is to prove the following 
stability theorem for regular solution, since the same occurs 
for weak solution using standard density arguments.

\begin{theorem}
\label{theorem:2}
With the regularity stated in Theorem~\ref{main results},  
the energy $E(t)$ decays exponentially as time approaches infinity, 
that is, there exists two positive constants, 
$\sigma_0$ and $\sigma_1$, such that
\begin{equation}
\label{energy:decay}	
E(t)\leqslant\sigma_0 e^{-\sigma_1 t}
\quad \forall t\geqslant 0.
\end{equation}
\end{theorem}

We begin by proving some important lemmas that will be essential 
to prove Theorem~\ref{theorem:2}.

\begin{lemma}
\label{lemma1:stab}
The functional 
\begin{equation*}
I_1(t)=\int_{0}^{L}\left(\rho u_t u+\dfrac{\mu}{2}u^2\right) dx
+\int_{0}^{L}\left(\rho_1\varphi_t \varphi+\dfrac{\gamma}{2}\varphi^2\right) dx
\end{equation*}
satisfies
\begin{equation}
\label{I'_1 ineq}	
\begin{split}
\dfrac{dI_1(t)}{dt}	
\leqslant &
-\alpha \int_{0}^{L} u_x^2 dx-\lambda \int_{0}^{L} (\varphi-u)^2 dx
- \dfrac{K}{2} \int_{0}^{L}(\varphi_x+\psi)^2 dx\\
&-\dfrac{b}{2}\int_{0}^{L} \psi^2_x dx +\rho\int_{0}^{L} u_t^2 dx
+\rho_1\int_{0}^{L} \varphi_t^2 dx +C_1\int_{0}^{L}w^2_{xt} dx.
\end{split}
\end{equation}
\end{lemma}

\begin{proof}
Multiplying  ${\text{(\ref{main:system})}}_\text{1}$  
by $u$, using the fact that $u_{tt}u
=\dfrac{\partial}{\partial t}(u_t u)-u_t^2$ 
and integrating by parts, we get that
\begin{equation}
\label{decay:eq1}	
\dfrac{d}{dt}\int_{0}^{L} \left(\rho u_{t} u+\dfrac{\mu}{2}u^2 \right) dx
-\rho\int_{0}^{L}u_t^2 dx +\alpha\int_{0}^{L}u_x^2 dx
- \lambda\int_{0}^{L} (\varphi-u)udx=0.
\end{equation}
Similarly, multiplying ${\text{(\ref{main:system})}}_\text{2}$ 
by $\varphi$, using the fact that $\varphi_{tt}\varphi
=\dfrac{\partial}{\partial t}(\varphi_t\varphi)-\varphi_t^2$ 
and integrating by parts, one obtains
\begin{equation}
\label{decay:eq2}	
\begin{split}
\dfrac{d}{dt}\int_{0}^{L} \left(\rho_1 \varphi_{t} 
\varphi+\dfrac{\gamma}{2} \varphi^2\right) dx
&-\rho_1\int_{0}^{L}\varphi_t^2 dx + K\int_{0}^{L} 
(\varphi_x+\psi)\varphi_xdx\\
&+ \lambda\int_{0}^{L} \left(\varphi-u\right)\varphi dx
+\beta\int_{0}^{L} w_{xt}\varphi dx=0.
\end{split}
\end{equation}
Now, multiplying ${\text{(\ref{main:system})}}_\text{3}$ by $\psi$, it results that
\begin{equation}
\label{decay:eq3}	
b\int_{0}^{L} \psi_x^2 dx 
+K \int_{0}^{L} (\varphi_x+\psi)\psi dx=0.
\end{equation} 
Adding (\ref{decay:eq1})--(\ref{decay:eq3}), we arrive at 
\begin{equation}
\label{I'_1 eq}	
\begin{split}
\dfrac{dI_1(t)}{dt}=&\rho\int_{0}^{L}u_t^2 dx 
-\alpha \int_{0}^{L}u_x^2 dx-\lambda 
\int_{0}^{L}(\varphi-u)^2 dx+\rho_1 \int_{0}^{L}\varphi_t^2dx\\
&-K  \int_{0}^{L} (\varphi_x+\psi)^2 dx-b\int_{0}^{L}\psi_x^2 dx  
+\beta\int_{0}^{L} w_t\varphi_x dx.
\end{split}
\end{equation} 
The fact that $\varphi_x=(\varphi_x+\psi)-\psi$ leads to
\begin{equation}
\label{phi_x:eq}	
\beta\int_{0}^{L} w_t\varphi_x dx
= \beta\int_{0}^{L} w_t(\varphi_x+\psi)dx
-\beta \int_{0}^{L}w_t \psi dx
\end{equation}
and Young's and Poincar\'{e}'s inequalities yield that
\begin{equation}
\label{I_1:eq1}	
\beta\int_{0}^{L} w_t(\varphi_x+\psi) dx
\leqslant \dfrac{K}{2}\int_{0}^{L}(\varphi_x+\psi)^2 dx 
+\dfrac{\beta^2c_p}{2K}\int_{0}^{L}w_{xt}^2 dx,
\end{equation} 
\begin{equation}
\label{I_1:eq2}	
-\beta\int_{0}^{L} w_t\psi dx
\leqslant \dfrac{b}{2}\int_{0}^{L}\psi_x^2 dx 
+\dfrac{\beta^2c^2_p}{2b}\int_{0}^{L}w_{xt}^2 dx,
\end{equation} 
where $c_p $ is a Poincar\'{e}'s constant. 
We obtain (\ref{I'_1 ineq}) by plugging 
(\ref{I_1:eq1}) and (\ref{I_1:eq2}) into (\ref{I'_1 eq}).
\end{proof}

\begin{lemma}
\label{lemma2:stab}
The functional 
\begin{equation*}
I_2(t)=\rho_3 \int_{0}^{L} w_t w dx
+\dfrac{\kappa}{2}\int_{0}^{L} w_x^2 dx+\beta\int_{0}^{L} \varphi_x w dx
\end{equation*}
satisfies
\begin{equation}
\label{I'_2(t):inequ}	
\dfrac{dI_2(t)}{dt}\leqslant  -\delta\int_{0}^{L}w_x^2dx
+\dfrac{b}{4}\int_{0}^{L} \psi_x^2dx+\dfrac{K}{4} 
\int_{0}^{L}(\varphi_x + \psi)^2dx
+C_2\int_{0}^{L} w_{xt}^2dx. 
\end{equation}
\end{lemma}

\begin{proof}
Differentiating $I_2$, using ${\text{(\ref{main:system})}}_\text{4}$,  
integration by parts and recalling the boundary conditions, we get that
\begin{equation}
\label{I'_2(t):eq}	
\dfrac{dI_2(t)}{dt} 
=\rho_3 \int_{0}^{L}w_{t}^2 dx -\delta \int_{0}^{L}w_{x}^2 dx
+\beta\int_{0}^{L}  w_t \varphi_{x} dx.
\end{equation}
It follows by (\ref{phi_x:eq}) and Young's and Poincar\'{e}'s inequalities that 
\begin{equation}
\label{I_2:eq1}	
\beta\int_{0}^{L} w_t(\varphi_x+\psi) dx
\leqslant \dfrac{K}{4}\int_{0}^{L}(\varphi_x+\psi)^2 dx 
+\dfrac{\beta^2c_p}{K}\int_{0}^{L}w_{xt}^2 dx
\end{equation}
and
\begin{equation}
\label{I_2:eq2}	
-\beta\int_{0}^{L} w_t\psi dx\leqslant \dfrac{b}{4}\int_{0}^{L}\psi_x^2 dx 
+\dfrac{\beta^2c^2_p}{b}\int_{0}^{L}w_{xt}^2 dx. 
\end{equation}
Applying Poincar\'{e}'s inequality leads to
\begin{equation}
\label{I_2:eq3}	
\rho_3\int_{0}^{L} w_t^2 dx
\leqslant \rho_3 c_p\int_{0}^{L}w_{xt}^2 dx. 
\end{equation}
By substituting (\ref{I_2:eq1}), (\ref{I_2:eq2}) 
and (\ref{I_2:eq3}) into (\ref{I'_2(t):eq}), 
we obtain that (\ref{I'_2(t):inequ}) holds.
\end{proof}	

Let us introduce now the Lyapunov functional
\begin{equation}
\label{Lyapunov:functional}	
L(t)=NE(t) + I_1(t)+ I_2(t),
\end{equation}
where $N$ is a positive constant to be fixed later.

\begin{lemma}
Let $(u,\varphi,\psi,w)$ be a solution of (\ref{main:system}). 
Then there exist two positive constants, $\xi_1$ and $\xi_2$, 
such that the functional energy $E$ is equivalent 
to functional $L$:
\begin{equation}
\label{Lyapunov:functional:ineq}	
\xi_1E(t)\leqslant L(t)\leqslant \xi_2E(t).
\end{equation}		
\end{lemma}

\begin{proof}
We know that	
\begin{equation*}
\begin{split}
\left| L(t)-NE(t)\right| 
\leqslant & \rho\int_{0}^{L} \left| u_t u\right|  dx
+\dfrac{\mu}{2} \int_{0}^{L} u^2 dx+\rho_1
\int_{0}^{L} \left| \varphi_t \varphi\right|  dx
+\dfrac{\gamma}{2}\int_{0}^{L} \varphi^2  dx\\ 
&+ \rho_3 \int_{0}^{L}\left|  w_t w\right| dx
+\dfrac{\kappa}{2}\int_{0}^{L}  w_x^2 dx
+\left| \beta\right| \int_{0}^{L}\left|  \varphi_x w\right| dx.
\end{split}
\end{equation*}
By using $\varphi_x=(\varphi_x+\psi)-\psi$,  
$u=-(\varphi-u)+\varphi$, and Young's and Poincar\'{e}s inequalities, 
we obtain, for some $\xi>0$, that
\begin{equation*}
\left|L(t)-NE(t)\right| \leqslant \xi E(t),  
\end{equation*}
which yields 
\begin{equation*}
(N-\xi)E(t)\leqslant L(t)\leqslant (N+\xi)E(t).
\end{equation*}
The estimate (\ref{Lyapunov:functional:ineq})
follows by choosing $N$ accordingly.
\end{proof}

We are now in conditions to prove Theorem~\ref{theorem:2}.

\begin{proof}(of Theorem~\ref{theorem:2})
Taking the time derivative of $L(t)$, using Lemmas~\ref{lemma1:stab} 
and \ref{lemma2:stab} and the energy dissipation law (\ref{dissipation eq}), 
it follows from (\ref{Lyapunov:functional}) that
\begin{equation*}
\begin{split}
\dfrac{dL(t)}{dt}
\leqslant& -\alpha\int_{0}^{L}u_x^2 dx -\lambda\int_{0}^{L}(\varphi-u)^2 dx\\
& -\dfrac{K}{4}\int_{0}^{L}(\varphi_x+\psi)^2 dx
-\dfrac{b}{4}\int_{0}^{L}\psi_x^2 dx-\delta \int_{0}^{1}w_x^2 dx\\
& -\left[ N\mu-\rho \right] \int_{0}^{L}u_t^2 dx
-\left[ N\gamma-\rho_1 \right] \int_{0}^{L}\varphi_t^2 dx\\
&-\left[N\kappa-C_1-C_2 \right] \int_{0}^{1}w_{xt}^2 dx.
\end{split}
\end{equation*}
Next, we choose $N$ large enough so that (\ref{Lyapunov:functional:ineq}) 
remains valid and 
\begin{equation*}
\begin{cases}
N\mu-\rho>0,\\
N\gamma-\rho_1>0,\\
N\kappa-C_1-C_2>0.
\end{cases}
\end{equation*}
Thus, for some  $\zeta_1>0$, we have 
\begin{equation*}
\dfrac{dL(t)}{dt}\leqslant -\zeta_1 
\int_{0}^{L}[u_t^2+u_x^2+(\varphi-u)^2
+\varphi_t^2+(\varphi_x+\psi)^2+\psi_x^2+w_{xt}^2+w_x^2]\; dx.
\end{equation*}
On account of (\ref{Energy funct}), we can write that
\begin{equation}
\label{L':ineq}	
\dfrac{dL(t)}{dt}\leqslant -\zeta_2 E(t),
\quad \forall\; t\geqslant0,
\end{equation}
for some  $\zeta_2>0$. Combining \eqref{L':ineq} 
with (\ref{Lyapunov:functional:ineq}), we get
\begin{equation}
\label{L':ineq2}	
\dfrac{dL(t)}{dt} \leqslant -\sigma_1 L(t), 
\quad \forall\; t\geqslant 0.
\end{equation}
A simple integration of (\ref{L':ineq2}) over $(0,t)$ leads to 
\begin{equation*}
L(t) \leqslant  L(0)e^{-\sigma_1 t}, 
\quad \forall\; t\geqslant 0. 
\end{equation*}
Again, if we recall (\ref{Lyapunov:functional:ineq}), 
the theorem is proved with $\sigma_0=\dfrac{\xi_2}{\xi_1}E(0)$.  	
\end{proof}


\section{Numerical approximation}
\label{sec:05}

Now, we present a numerical analysis of the problem 
studied theoretically in Sections~\ref{sec:03} and \ref{sec:04}.


\subsection{Description of the discrete problem}

We acquire a weak form associated to the continuous problem
by multiplying equations (\ref{main:system}) with the test 
functions $\bar{u}$, $\bar{\varphi}$, $\bar{\psi}$, $\bar{w} \in H^1_0(0,L)$, 
respectively. Let $\xi=u_t$, $\Phi=\varphi_t$, $\Psi=\psi_t$, 
and $\vartheta=w_t$. Applying integration by parts 
and using the boundary conditions, we find that
\begin{equation}
\label{variational:problem}	
\begin{cases}
\rho( \xi_t,\bar{u})+\alpha(u_x,\bar{u}_x)
-\lambda(\varphi-u,\bar{u})+\mu(\xi,\bar{u})=0,\\
\rho_1(\Phi_t,\bar{\varphi})+K(\varphi_x+\psi,\bar{\varphi}_x)
+\lambda(\varphi-u,\bar{\varphi})+\gamma(\Phi,\bar{\varphi})
+\beta(\vartheta_x,\bar{\varphi})=0,\\
b(\psi_x,\bar{\psi}_x)+K(\varphi_x+\psi,\bar{\psi})=0,\\
\rho_3(\vartheta_t,\bar{w})+\delta(w_x,\bar{w}_x)
+\beta(\Phi_x,\bar{w})+\kappa(\vartheta_x,\bar{w}_x)=0.
\end{cases}
\end{equation}
To obtain the spatial approximation, we introduce a uniform 
partition $(\varGamma_h)_h$ of the interval $[0, L]$
into $M$ subintervals, such that $0 = x_0 < x_1 < \cdots < x_M = L$, 
with a uniform length $h = x_{i+1} - x_i = 1/M$. Therefore,
the variational space $H^1_0$ is approximated by the 
finite-dimensional space $S^h$, defined as
\begin{equation*}
S_h=\left\lbrace v_h \in H^1_0(0,L): \forall \; 
[x_i,x_{i+1}] \in \varGamma_h,\; 
v_h|_{[x_i,x_{i+1}]} \in P_1([x_i,x_{i+1}]) \right\rbrace.
\end{equation*}
Here, the space $P_1([x_i,x_{i+1}])$ denotes the space of polynomials 
of a degree $\leqslant 1$ in the subinterval $[x_i,x_{i+1}]$, i.e., 
the finite element space is made of continuous and piecewise affine functions,  
and $h$ is the spatial discretization parameter.

In order to define the discrete initial conditions, 
assuming that they are smooth enough, we set
\begin{equation*}
u^0_h=P_h u_0,\;\xi^0_h=P_h u_1,\;\varphi^0_h=P_h \varphi_0,\; 
\Phi^0_h=P_h \varphi_1,\;\psi^0_h=P_h \psi_0,\; 
w^0_h=P_h w_0,\; \vartheta^0_h=P_h w_1,
\end{equation*}
where $P_h$ is the projection operator,  
$P_h: H^1_0(0,L)\longrightarrow S_h$ (see \cite{Bernardi}), 
defined by
\begin{equation*}
\forall \; \chi_h \in S_h \;\; ((P_h\eta-\eta)_x,\chi_{hx})=0.
\end{equation*}
The operator $P_h$  preserves the values at all end points of 
the elements in $\Gamma_h$ and fulfills the following 
estimate for all $\eta \in H^1_0(0,L)$:
\begin{equation*}
\rVert P_h \eta-\eta\lVert\leqslant Ch\rVert\eta_x\lVert,
\end{equation*}
as well as for more regular functions $\eta$,
\begin{equation}
\label{space:error}	
\rVert P_h \eta-\eta\lVert +h\rVert (P_h \eta-\eta)_x\lVert 
\leqslant Ch^2\rVert\eta_{xx}\lVert.
\end{equation} 

As a second step, to discretize the time derivatives for a given final time $T > 0$ 
and a given positive integer $N$, we define the time step $\Delta t = T/N$ 
and the nodes $t_n = n\Delta t$, $n = 0, \ldots , N$.

When using the backward Euler scheme in time, the fully finite element 
approximation of the variational problem (\ref{variational:problem}) 
consists to find $\xi^n_h,\Phi^n_h,\psi^n_h,\vartheta^n_h \in S_h$ 
such that, for $n=1,\ldots, N$ and for all $\bar{u}_h$, $\bar{\varphi}_h$, 
$\bar{\psi}_h$, $\bar{w}_h \in S_h$,
\begin{equation}
\label{Stability:eq1}	
\begin{cases}
\dfrac{\rho}{\Delta t}( \xi_h^n -\xi_h^{n-1},\bar{u}_h)+\alpha(u^n_{hx},\bar{u}_{hx})
-\lambda(\varphi^n_h-u^n_h,\bar{u}_h)+\mu(\xi^n_h,\bar{u}_h)=0,\\[0.3cm]
\begin{split}
\dfrac{\rho_1}{\Delta t}
&(\Phi_h^n-\Phi_h^{n-1},\bar{\varphi}_h)+K(\varphi_{hx}^n+\psi^n_h,\bar{\varphi}_{hx})\\
&+\lambda(\varphi_h^n-u_h^n,\bar{\varphi_h})
+\gamma(\Phi^n_h,\bar{\varphi}_h)+\beta(\vartheta^n_{hx},\bar{\varphi}_h)=0,
\end{split}\\[0.3cm]
b(\psi_{hx}^n,\bar{\psi}_{hx})+K(\varphi_{hx}^n+\psi^n_h,\bar{\psi}_h)=0,\\[0.3cm]
\dfrac{\rho_3}{\Delta t}(\vartheta_h^n-\vartheta_h^{n-1},\bar{w}_h)
+\delta(w^n_{hx},\bar{w}_{hx})+\beta(\Phi^n_{hx},\bar{w}_h)
+\kappa(\vartheta^n_{hx},\bar{w}_{hx})=0,
\end{cases}
\end{equation}
where
\begin{equation*}
u^n_h= u^{n-1}_h+\Delta t\; \xi^n_h,\;\varphi^n_h
= \varphi^{n-1}_h+\Delta t\; \Phi^n_h,\; 
w^n_h= w^{n-1}_h+\Delta t\; \vartheta^n_h.
\end{equation*}
By using the well-known Lax--Milgram lemma and the assumptions imposed on the
constitutive parameters, it is easy to obtain that the fully discrete 
problem (\ref{Stability:eq1}) has a unique solution.


\subsection{Study of the discrete energy} 

The next result is a discrete version of the energy decay property 
(\ref{dissipation eq}) satisfied by the continuous solution.

\begin{theorem}
Let the discrete energy be given by
\begin{equation}
\label{discrete:energy}	
\begin{split}
E^n=&\dfrac{1}{2}\left(\rho\rVert \xi^n_h\lVert^2
+\alpha \rVert u^n_{hx}\lVert^2+\lambda\rVert \varphi^n_h
- u^n_{h}\lVert^2+\rho_1\rVert \Phi^n_h\lVert^2\right)\\
&+\dfrac{1}{2}\left( K \rVert \varphi^n_{hx}+\psi^n_h\lVert^2
+b \rVert \psi^n_{hx}\lVert^2+\rho_3\rVert \vartheta^n_h\lVert^2
+\delta\rVert w^n_{hx}\lVert^2 \right). 
\end{split}
\end{equation}
Then, the decay property
\begin{equation*}
\dfrac{E^n-E^{n-1}}{\Delta t}\leqslant 0,
\end{equation*}
holds for $n=1,2,\ldots,N$.
\end{theorem}

\begin{proof}
Taking $\bar{u}_h=\xi^n_h$, 
$\bar{\varphi}_h=\Phi^n_h$, $\bar{\psi}_h=\Psi^n_h$, 
and $\bar{w}_h=\vartheta^n_h$ in (\ref{Stability:eq1}), 
it results that
\begin{equation}
\label{Stability:eq3}	
\dfrac{\rho}{2\Delta t}\left(\rVert \xi^n_h-\xi^{n-1}_h \lVert^2
+\rVert \xi^n_h \lVert^2-\rVert \xi^{n-1}_h\lVert^2\right)
+\alpha(u^n_{hx},\xi^n_{hx})-\lambda(\varphi^n_h-u^n_h,\xi^n_h)
+\mu{\lVert\xi^n_h\rVert}^2=0,
\end{equation}
\begin{equation}
\label{Stability:eq4}	
\begin{split}
\dfrac{\rho_1}{2\Delta t}
&\left(\rVert \Phi^n_h-\Phi^{n-1}_h \lVert^2
+\rVert \Phi^n_h \lVert^2-\rVert \Phi^{n-1}_h
\lVert^2 \right)+K(\varphi_{hx}^n+\psi^n_h,\Phi^n_{hx})\\
&+\lambda(\varphi_{h}^n-u^n_h,\Phi^n_{h})
+\gamma{\lVert\Phi^n_h\rVert}^2+\beta(\vartheta^n_{hx},\Phi^n_h)=0,
\end{split}
\end{equation}
\begin{equation}
\label{Stability:eq5}	
b(\psi_{hx}^n,\Psi^n_{hx})+K(\varphi_{hx}^n+\psi^n_h,\Psi^n_h)=0,
\end{equation}
and
\begin{equation}
\label{Stability:eq6}	
\dfrac{\rho_3}{2\Delta t}\left(\rVert \vartheta^n_h
-\vartheta^{n-1}_h \lVert^2+\rVert \vartheta^n_h \lVert^2
-\rVert \vartheta^{n-1}_h\lVert^2 \right)+\delta(w^n_{hx},
\vartheta^n_{hx})+\beta(\Phi^n_{hx},\vartheta^n_h)
+\kappa{\lVert\vartheta^n_{hx}\rVert}^2=0.
\end{equation}
Summing equations (\ref{Stability:eq3})--(\ref{Stability:eq6}) 
and keeping in mind that 
\begin{equation*}
K(\varphi_{hx}^n+\psi^n_h,\Phi^n_{hx}+\Psi^n_h)
\geqslant\dfrac{K}{2\Delta t} \left(\rVert \varphi_{hx}^n
+\psi^n_h \lVert^2-\rVert \varphi_{hx}^{n-1}
+\psi^{n-1}_h \lVert^2 \right),
\end{equation*}
\begin{equation*}
\lambda(\varphi_{h}^n-u^n_h,\Phi^n_{hx}-\xi^n_h)
\geqslant\dfrac{\lambda}{2\Delta t} \left(\rVert \varphi_{h}^n
-u^n_h \lVert^2-\rVert \varphi_{h}^{n-1}-u^{n-1}_h \lVert^2 \right),
\end{equation*}
\begin{equation*}
\alpha(u_{hx}^n,\xi^n_{hx})\geqslant\dfrac{\alpha}{2\Delta t} 
\left(\rVert u_{hx}^n \lVert^2-\rVert u_{hx}^{n-1} \lVert^2 \right), 
\end{equation*}
\begin{equation*}
b(\psi_{hx}^n,\Psi^n_{hx})\geqslant\dfrac{b}{2\Delta t} 
\left(\rVert \psi_{hx}^n \lVert^2-\rVert \psi_{hx}^{n-1} \lVert^2 \right),
\end{equation*}
and
\begin{equation*}
\delta(w_{hx}^n,\vartheta^n_{hx})\geqslant\dfrac{\delta}{2\Delta t} 
\left(\rVert w_{hx}^n \lVert^2-\rVert w_{hx}^{n-1} \lVert^2 \right),
\end{equation*}
we find
\begin{equation*}
\begin{split}
\dfrac{\rho}{2\Delta t}
&\left(\rVert \xi^n_h-\xi^{n-1}_h \lVert^2+\rVert \xi^n_h \lVert^2
-\rVert \xi^{n-1}_h\lVert^2 \right)+\dfrac{\alpha}{2\Delta t} 
\left(\rVert u_{hx}^n \lVert^2-\rVert u_{hx}^{n-1} \lVert^2 \right)\\
&+ \dfrac{\lambda}{2\Delta t} \left(\rVert \varphi_{h}^n-u^n_h \lVert^2
-\rVert \varphi_{h}^{n-1}-u^{n-1}_h \lVert^2 \right)+\mu \rVert \xi^n_h\lVert^2\\
&+\dfrac{\rho_1}{2\Delta t}\left(\rVert \Phi^n_h-\Phi^{n-1}_h \lVert^2
+\rVert \Phi^n_h \lVert^2-\rVert \Phi^{n-1}_h\lVert^2 \right)\\
&+\dfrac{K}{2\Delta t} \left(\rVert \varphi_{hx}^n+\psi^n_h \lVert^2
-\rVert \varphi_{hx}^{n-1}+\psi^{n-1}_h \lVert^2 \right)
+\gamma \rVert \Phi^n_h\lVert^2\\
&+\dfrac{\rho_3}{2\Delta t}\left(\rVert \vartheta^n_h
-\vartheta^{n-1}_h \lVert^2+\rVert \vartheta^n_h \lVert^2
-\rVert \vartheta^{n-1}_h\lVert^2 \right)
+\kappa \rVert \vartheta^n_{hx}\lVert^2 \\
&+\dfrac{b}{2\Delta t} \left(\rVert \psi_{hx}^n 
\lVert^2-\rVert \psi_{hx}^{n-1} \lVert^2 \right)
+\dfrac{\delta}{2\Delta t} \left(\rVert w_{hx}^n 
\lVert^2-\rVert w_{hx}^{n-1} \lVert^2 \right) \leqslant 0.
\end{split}
\end{equation*}	
Note that $\rVert \xi^n_h-\xi^{n-1}_h \lVert^2$, 
$\rVert \Phi^n_h-\Phi^{n-1}_h \lVert^2$,  
$\rVert \vartheta^n_h-\vartheta^{n-1}_h \lVert^2$,  
$\rVert \xi^n_h \lVert^2$, $\rVert \Phi^n_h\lVert^2$, 
and $\rVert \vartheta^n_{hx}\lVert^2$  
are positive terms. We deduce that
\begin{equation*}
\begin{split}
&\dfrac{\rho}{2\Delta t}\left(\rVert \xi^n_h \lVert^2
-\rVert\xi^{n-1}_h\lVert^2 \right)+\dfrac{\alpha}{2\Delta t} 
\left(\rVert u_{hx}^n \lVert^2-\rVert u_{hx}^{n-1} \lVert^2 \right)\\
&+\dfrac{\lambda}{2\Delta t} \left(\rVert \varphi_{h}^n
-u^n_h \lVert^2-\rVert \varphi_{h}^{n-1}-u^{n-1}_h \lVert^2 \right)
+\dfrac{\rho_1}{2\Delta t}\left(\rVert \Phi^n_h \lVert^2
-\rVert\Phi^{n-1}_h\lVert^2 \right)\\
&+\dfrac{K}{2\Delta t} \left(\rVert \varphi_{hx}^n+\psi^n_h \lVert^2
-\rVert \varphi_{hx}^{n-1}+\psi^{n-1}_h \lVert^2 \right)
+\dfrac{b}{2\Delta t} \left(\rVert \psi_{hx}^n \lVert^2
-\rVert \psi_{hx}^{n-1} \lVert^2 \right)\\
& + \dfrac{\rho_3}{2\Delta t}\left(\rVert \vartheta^n_h \lVert^2
-\rVert \vartheta^{n-1}_h\lVert^2 \right)
+\dfrac{\delta}{2\Delta t}\left(\rVert w^n_{hx} \lVert^2
-\rVert w^{n-1}_{hx}\lVert^2 \right) \leqslant 0,
\end{split}
\end{equation*}	
which proves the intended result.
\end{proof}


\subsection{Error estimate}

We now state and prove some a priori error estimates for 
the difference between the exact solution and the numerical solution.

The linear convergence of the numerical method 
is outlined in the following theorem.

\begin{theorem}
Suppose that the solution to the continuous problem 
(\ref{main:system}) is regular enough, that is,
\begin{equation*}
u,\varphi,w \in H^3(0,T;L^2(0,L))
\cap H^2(0,T;H^1(0,L))\cap W^{1,\infty}(0,T;H^2(0,L)),
\end{equation*}
\begin{equation*}
\psi \in H^1(0,T;H^2(0,L)).
\end{equation*}
Then, the following error estimates hold:
\begin{equation*}
\begin{split}
&\rVert \xi^n_h-\xi(t_n)\lVert^2+\rVert u^n_{hx}-(u(t_n))_x\lVert^2
+ \rVert \varphi^n_{h}-u^n_h-(\varphi(t_n)-u(t_n))\lVert^2\\
&+\rVert \Phi^n_h-\Phi(t_n)\lVert^2+\rVert \varphi^n_{hx}
+\psi^n_h-((\varphi(t_n))_x+\psi(t_n))\lVert^2 
+\rVert \psi^n_{hx}-(\psi(t_n))_x\lVert^2\\
&+\rVert \vartheta^n_{h}-\vartheta(t_n)\lVert^2 
+\rVert w^n_{hx}-(w(t_n))_x\lVert^2\leqslant C(h^2+(\Delta t)^2),
\end{split}
\end{equation*}
where $C$ is independent of $\Delta t$ and $h$.
\end{theorem}

\begin{proof}
As a first step, let us define
\begin{equation*}
z^n=u^n_h-P_h u(t_n),\;\hat{z}^n=\xi^n_h-P_h\xi(t_n),
\end{equation*}
\begin{equation*}
e^n=\varphi^n_h-P_h \varphi(t_n),\;\hat{e}^n=\Phi^n_h-P_h\Phi(t_n),
\end{equation*}
\begin{equation*}
y^n=\psi^n_h-P_h \psi(t_n),\;\hat{y}^n=\Psi^n_h-P_h \Psi(t_n),
\end{equation*}
and
\begin{equation*}
\varrho^n=w^n_h-P_h w(t_n),\;\hat{\varrho}^n=\vartheta^n_h-P_h \vartheta(t_n).
\end{equation*}
Substituting in the scheme (\ref{Stability:eq1}) 
and taking $\bar{u}_h=\hat{z}^n$, $\bar{\varphi}_h=\hat{e}^n$,
$\bar{\psi}_h=\hat{y}^n$, and $\bar{w}=\hat{\varrho}^n$, all this gives
\begin{equation}
\label{step1:eq1}	
\begin{split}
\dfrac{\rho}{2\Delta t}
&\left(\rVert \hat{z}^n-\hat{z}^{n-1} \lVert^2
+\rVert \hat{z}^n \lVert^2-\rVert \hat{z}^{n-1}\lVert^2 \right)
+\dfrac{\rho}{\Delta t}(P_h\xi(t_n)-P_h\xi(t_{n-1}),\hat{z}^n)
+\alpha(z^n_x,\hat{z}^n_x)\\
&+\alpha((P_h u(t_n))_x,\hat{z}^n_x)-\lambda(e^n-z^n,\hat{z}^n)
-\lambda(P_h\varphi(t_n)-P_h u(t_n),\hat{z}^n)
+ \mu\rVert \hat{z}^n \lVert^2\\&+\mu(P_h\xi(t_n),\hat{z}^n)=0, 
\end{split}
\end{equation}
\begin{equation}
\begin{split}
\dfrac{\rho_1}{2\Delta t}
&\left(\rVert \hat{e}^n-\hat{e}^{n-1} \lVert^2
+\rVert \hat{e}^n \lVert^2-\rVert \hat{e}^{n-1} \lVert^2 \right)
+\dfrac{\rho_1}{\Delta t}(P_h\Phi(t_n)-P_h\Phi(t_{n-1}),\hat{e}^n)\\
&+K(e^n_x+y^n,\hat{e}_{x}^n)+K((P_h\varphi(t_n))_x+P_h\psi(t_n),\hat{e}_{x}^n)
+\lambda(e^n-z^n,\hat{e}^n)\\
&+\lambda(P_h\varphi(t_n)-P_h u(t_n),\hat{e}^n)
+ \gamma\rVert \hat{e}^n \lVert^2+\gamma(P_h\Phi (t_n),
\hat{e}^n)+\beta(\hat{\varrho}^n_x,\hat{e}^n)\\
&+\beta((P_h\vartheta (t_n))_x,\hat{e}^n)=0, 
\end{split}
\end{equation}
\begin{equation}
 b(y^n_x ,\hat{y}^{n}_x)+b((P_h \psi(t_n))_x,\hat{y}^{n}_x)
+K(e^n_x+y^n,\hat{y}^n)+K((P_h\varphi(t_n))_x
+P_h\psi(t_n),\hat{y}^n)=0,
\end{equation}
\begin{equation}
\label{step1:eq2}	
\begin{split}
\dfrac{\rho_3}{2\Delta t}
&\left(\rVert \hat{\varrho}^n-\hat{\varrho}^{n-1} \lVert^2
+\rVert \hat{\varrho}^n \lVert^2-\rVert 
\hat{\varrho}^{n-1} \lVert^2 \right)
+\dfrac{\rho_3}{\Delta t}(P_h\vartheta(t_n)
-P_h\vartheta(t_{n-1}),\hat{\varrho}^n)\\
&+\delta (\varrho^n_x,\hat{\varrho}^n_x)
+\delta ((P_h w(t_n))_x,\hat{\varrho}^n_x)
+\beta(\hat{e}^n_x,\hat{\varrho}^n)
+\beta((P_h\Phi(t_n))_x,\hat{\varrho}^n)\\
&+\kappa \rVert \hat{\varrho}^n_x \lVert^2
+\kappa((P_h\vartheta(t_n))_x,\hat{\varrho}^n_x)=0.
\end{split} 
\end{equation}
Let $\bar{u}=\hat{z}^n$, 
$\bar{\varphi}=\hat{e}^n$, $\bar{\psi}=\hat{y}^n$,
$\bar{w}=\varrho^n$,
in the weak form (\ref{variational:problem}).
We combine the resulting equations with 
(\ref{step1:eq1})--(\ref{step1:eq2}) to obtain
\begin{equation}
\label{step2:eq1}	
\begin{split}
\dfrac{\rho}{2\Delta t}
&\left(\rVert \hat{z}^n-\hat{z}^{n-1} \lVert^2
+\rVert \hat{z}^n \lVert^2-\rVert \hat{z}^{n-1} \lVert^2 \right)
+\alpha(z^n_x,\hat{z}^n_x)-\lambda(e^n-z^n,\hat{z}^n)
+\mu \rVert \hat{z}^n \lVert^2\\
=&\rho(\xi_t(t_n)-\dfrac{P_h\xi(t_n)-P_h\xi(t_{n-1})}{\Delta t},
\hat{z}^n)+\alpha(u_x(t_n)-(P_hu(t_n))_x,\hat{z}^n_x)\\
&-\lambda(\varphi(t_n)-u(t_n)-(P_h\varphi(t_n)-P_h u(t_n)),
\hat{z}^n)+\mu(\xi(t_n)-P_h\xi(t_n),\hat{z}^n),
\end{split}
\end{equation}
\begin{equation}
\begin{split}
\dfrac{\rho_1}{2\Delta t}
&\left(\rVert \hat{e}^n-\hat{e}^{n-1} \lVert^2
+\rVert \hat{e}^n \lVert^2-\rVert \hat{e}^{n-1} \lVert^2 \right)
+K(e^n_x+y^n,\hat{e}_{x}^n)+\lambda(e^n-z^n,\hat{e}^n)
\\&+\gamma\rVert \hat{e}^n \lVert^2+\beta(\hat{\varrho}^n_x,\hat{e}^n)
=\rho_1(\Phi_t(t_n)-\dfrac{P_h\Phi(t_n)-P_h\Phi(t_{n-1})}{\Delta t},
\hat{e}^n)\\&+K(\varphi_x(t_n)+\psi(t_n)-((P_h\varphi(t_n))_x
+P_h\psi(t_n)),\hat{e}_{x}^n)\\
&+\lambda(\varphi(t_n)-u(t_n) -(P_h\varphi(t_n)-P_hu(t_n)),\hat{e}^n)\\
&+\gamma(\Phi(t_n)-P_h  \Phi(t_n),\hat{e}^n)+\beta(\vartheta_x(t_n)
-(P_h\vartheta(t_n))_x,\hat{e}^n),
\end{split}
\end{equation}
\begin{equation}
\begin{split} 
b(y^n_x ,\hat{y}^{n}_x)+K(e^n_x+y^n,\hat{y}^n)
=& b(\psi_x(t_n)-(P_h \psi(t_n))_x,\hat{y}^{n}_x)\\
&+K(\varphi_x(t_n)+\psi(t_n)-((P_h\varphi(t_n))_x
+P_h\psi(t_n)),\hat{y}^n),
\end{split}
\end{equation}
\begin{equation}
\label{step2:eq2}	
\begin{split}
\dfrac{\rho_3}{2\Delta t}
&\left(\rVert \varrho^n-\varrho^{n-1} \lVert^2
+\rVert \varrho^n \lVert^2-\rVert \varrho^{n-1} \lVert^2 \right)
+\delta(\varrho^n_x,\hat{\varrho}^n_x)+\beta(\hat{e}^n_x,
\hat{\varrho}^n)+\kappa\rVert \hat{\varrho}^n_x \lVert^2\\
=& \rho_3(\vartheta_t(t_n)-\dfrac{P_h\vartheta(t_n)-P_h 
\vartheta(t_{n-1})}{\Delta t},\hat{\varrho}^n)
+\delta(w_x(t_n)-(P_h w(t_n))_x,\hat{\varrho}^n_x)\\
&+ \beta(\Phi_x(t_n)-(P_h \Phi(t_n))_x,\hat{\varrho}^n)
+\kappa(\vartheta_x(t_n)-(P_h \vartheta(t_n))_x,\hat{\varrho}^n_x).
\end{split}
\end{equation}
We sum up the last four equations, to obtain that
\begin{equation}
\label{step3}	
\begin{split}
&\dfrac{\rho}{2\Delta t}\left(\rVert \hat{z}^n
-\hat{z}^{n-1} \lVert^2+\rVert \hat{z}^n \lVert^2
-\rVert \hat{z}^{n-1} \lVert^2 \right)+\alpha(z^n_x,\hat{z}^n_x)
+\lambda(e^n-z^n,\hat{e}^n-\hat{z}^n)+\mu\rVert \hat{z}^n \lVert^2\\
&+\dfrac{\rho_1}{2\Delta t}\left(\rVert \hat{e}^n
-\hat{e}^{n-1} \lVert^2+\rVert \hat{e}^n \lVert^2
-\rVert \hat{e}^{n-1} \lVert^2 \right)
+K(e^n_x+y^n,\hat{e}_{x}^n+\hat{y}^n)
+\gamma \rVert \hat{e}^{n} \lVert^2\\
&+b(y^n_x ,\hat{y}^{n}_x)+\dfrac{\rho_3}{2\Delta t}\left(\rVert 
\hat{\varrho}^n-\hat{\varrho}^{n-1} \lVert^2
+\rVert \hat{\varrho}^n \lVert^2-\rVert \hat{\varrho}^{n-1} 
\lVert^2 \right)+\delta(\varrho^n_x,\hat{\varrho}^n_x)
+\kappa\rVert \hat{\varrho}^n_x \lVert^2 \\[0.2cm]
=& \rho(\xi_t(t_n)-\dfrac{P_h\xi(t_n)-P_h\xi(t_{n-1})}{\Delta t},
\hat{z}^n)+\mu(\xi(t_n)-P_h\xi(t_n),\hat{z}^n)\\
& +\alpha(u_x(t_n)-(P_hu(t_n))_x,\hat{z}^n_x)
+\lambda(\varphi(t_n)-u(t_n)-(P_h\varphi(t_n)
-P_h u(t_n)),\hat{e}^n-\hat{z}^n)\\
&+\rho_1(\Phi_t(t_n)-\dfrac{P_h\Phi(t_n)-P_h\Phi(t_{n-1})}{\Delta t},
\hat{e}^n)+\gamma(\Phi(t_n)-P_h  \Phi(t_n),\hat{e}^n)\\
&+K(\varphi_x(t_n)+\psi(t_n)-((P_h\varphi(t_n))_x+P_h\psi(t_n)),
\hat{e}_{x}^n+\hat{y}^n)\\
&+b(\psi_x(t_n)-(P_h \psi(t_n))_x,\hat{y}^{n}_x)
+\rho_3(\vartheta_t(t_n)-\dfrac{P_h \vartheta(t_n)
-P_h \vartheta(t_{n-1})}{\Delta t},\hat{\varrho}^n)\\
&+\delta(w_x(t_n)-(P_h w(t_n))_x,\hat{\varrho}^n_x)
+\kappa(\vartheta_x(t_n)-(P_h \vartheta(t_n))_x,\hat{\varrho}^n_x).
\end{split}
\end{equation}
Now, by using the definition of $\hat{z}^n$, $\hat{e}^n$,  
$\hat{y}^n$ and $\hat{\varrho}^n$, we obtain the following estimates:
\begin{equation}
\label{step4:eq1}	
\begin{split}
(e^n-z^n,\hat{e}^n-\hat{z}^n)
&=( e^n-z^n,\dfrac{e^n-e^{n-1}}{\Delta t}-\dfrac{z^n-z^{n-1}}{\Delta t}) \\
&\;\;\;+( e^n-z^n,\dfrac{P_h\varphi(t_n)-P_h\varphi(t_{n-1})}{\Delta t}-P_h\Phi(t_n)) \\
&\;\;\;-( e^n-z^n,\dfrac{P_h u(t_n)-P_h u(t_{n-1})}{\Delta t}-P_h \xi(t_n))\\
&=\dfrac{1}{2\Delta t}\left(  \rVert e^n-z^n-(e^{n-1}-z^{n-1})\lVert^2
+\rVert e^n-z^n\lVert^2-\rVert e^{n-1}-z^{n-1}\lVert^2\right)\\
&\;\;\;+( e^n-z^n,\dfrac{P_h\varphi(t_n)-P_h\varphi(t_{n-1})}{\Delta t}-P_h\Phi(t_n))\\
&\;\;\;-( e^n-z^n,\dfrac{P_h u(t_n)-P_h u(t_{n-1})}{\Delta t}-P_h \xi(t_n)),
\end{split}
\end{equation}
\begin{equation}
\label{step4:eq2}	
\begin{split}
(e^n_x+y^n,\hat{e}_{x}^n+\hat{y}^n)
&=( e^n_x+y^n,\dfrac{e^n_x-e^{n-1}_x}{\Delta t}+\dfrac{y^n-y^{n-1}}{\Delta t})\\
&\;\;\;+( e^n_x+y^n,\dfrac{(P_h\varphi(t_n))_x
-(P_h\varphi(t_{n-1}))_x}{\Delta t}-(P_h\Phi(t_n))_x) \\
&\;\;\;+( e^n_x+y^n,\dfrac{P_h\psi(t_n)-P_h\psi(t_{n-1})}{\Delta t}-P_h\Psi(t_n))\\
&=\dfrac{1}{2\Delta t}\left(  \rVert e^n_x+y^n-(e^{n-1}_x+y^{n-1})
\lVert^2+\rVert e^n_x+y^n\lVert^2-\rVert e^{n-1}_x+y^{n-1}\lVert^2\right) \\
&\;\;\;+( e^n_x+y^n,\dfrac{(P_h\varphi(t_n))_x
-(P_h\varphi(t_{n-1}))_x}{\Delta t}-(P_h\Phi(t_n))_x) \\
&\;\;\;+( e^n_x+y^n,\dfrac{P_h\psi(t_n)-P_h\psi(t_{n-1})}{\Delta t}-P_h\Psi(t_n)),
\end{split}
\end{equation}
\begin{equation}
\label{step4:eq3}	
\begin{split}
(z^n_x,\hat{z}^n_x)
&=\dfrac{1}{2\Delta t}\left(\rVert z^n_x-z^{n-1}_x \lVert^2
+\rVert z^n_x \lVert^2-\rVert z^{n-1}_x \lVert^2 \right)\\
&\;\;\;+(z^n_x,\dfrac{(P_h u(t_n))_x
-(P_h u(t_{n-1}))_x}{\Delta t}-(P_h \xi(t_n))_x),
\end{split}
\end{equation}
\begin{equation}
\label{step4:eq4}	
\begin{split}
(y^n_x,\hat{y}^n_x)
&=\dfrac{1}{2\Delta t}\left(\rVert y^n_x-y^{n-1}_x \lVert^2
+\rVert y^n_x \lVert^2-\rVert y^{n-1}_x \lVert^2 \right)\\
&\;\;\;+(y^n_x,\dfrac{(P_h\psi(t_n))_x
-(P_h\psi(t_{n-1}))_x}{\Delta t}-(P_h\Psi(t_n))_x),
\end{split}
\end{equation}
and
\begin{equation}
\label{step4:eq}	
\begin{split}
(\varrho^n_x,\hat{\varrho}^n_x)
&=\dfrac{1}{2\Delta t}\left(\rVert \varrho^n_x-\varrho^{n-1}_x 
\lVert^2+\rVert \varrho^n_x \lVert^2
-\rVert \varrho^{n-1}_x \lVert^2 \right)\\
&\;\;\;+(\varrho^n_x,\dfrac{(P_h w(t_n))_x
-(P_h w(t_{n-1}))_x}{\Delta t}-(P_h \vartheta(t_n))_x).
\end{split}
\end{equation}
Inserting (\ref{step4:eq1})--(\ref{step4:eq}) into (\ref{step3}), 
then keeping in mind that $\rVert \hat{z}^n-\hat{z}^{n-1} \lVert^2$, 
$\rVert z^n_x- z^{n-1}_x \lVert^2$, 
$\rVert e^n-z^n-(e^{n-1}-z^{n-1})\lVert^2$, 
$\rVert \hat{e}^n-\hat{e}^{n-1} \lVert^2$, 
$\rVert e^n_x+y^n-(e^{n-1}_x+y^{n-1})\lVert^2$,
$\rVert y^n_x-y^{n-1}_x \lVert^2$,
$\rVert \hat{\varrho}^n- \hat{\varrho}^{n-1} \lVert^2$,
$\rVert \varrho^n_x-\varrho^{n-1}_x \lVert^2$,
$\rVert \hat{z}^n \lVert^2$, $\rVert \hat{e}^n \lVert^2$, 
and  $\rVert \hat{\varrho}^n_x \lVert^2$
are positive terms, we arrive at
\pagebreak
\begin{equation*}
\label{step4}	
\begin{split}
&\dfrac{\rho}{2\Delta t}\left(\rVert \hat{z}^n \lVert^2
-\rVert \hat{z}^{n-1} \lVert^2 \right) 
+\dfrac{\alpha}{2\Delta t}\left(\rVert z^n_x \lVert^2
-\rVert z^{n-1}_x \lVert^2 \right)
+\dfrac{\lambda}{2\Delta t}\left(\rVert e^n-z^n\lVert^2
-\rVert e^{n-1}-z^{n-1}\lVert^2\right)\\
&+\dfrac{\rho_1}{2\Delta t}\left(\rVert \hat{e}^n \lVert^2
-\rVert \hat{e}^{n-1} \lVert^2 \right)
+\dfrac{K}{2\Delta t}\left(\rVert e^n_x+y^n\lVert^2
-\rVert e^{n-1}_x+y^{n-1}\lVert^2\right)
+\dfrac{b}{2\Delta t}\left(\rVert y^n_x \lVert^2
-\rVert y^{n-1}_x \lVert^2 \right)\\
&+\dfrac{\rho_3}{2\Delta t}\left(\rVert \hat{\varrho}^n \lVert^2
-\rVert \hat{\varrho}^{n-1} \lVert^2 \right)
+\dfrac{\delta}{2\Delta t}\left(\rVert \varrho^n_x \lVert^2
-\rVert \varrho^{n-1}_x \lVert^2 \right)\\[0.2cm]
\leqslant &\rho(\xi_t(t_n)-\dfrac{P_h\xi(t_n)-P_h \xi(t_{n-1})}{\Delta t},
\hat{z}^n)+\alpha(u_x(t_n)-(P_hu(t_n))_x,\hat{z}^n_x)\\
&-\alpha(z^n_x,\dfrac{(P_h u(t_n))_x
-(P_h u(t_{n-1}))_x}{\Delta t}-(P_h \xi(t_n))_x)
+\lambda(\varphi(t_n)-u(t_n)-(P_h\varphi(t_n)-P_h u(t_n)),\hat{e}^n-\hat{z}^n)\\
&-\lambda( e^n-z^n,\dfrac{P_h\varphi(t_n)
-P_h\varphi(t_{n-1})}{\Delta t}-P_h\Phi(t_n))
+\lambda( e^n-z^n,\dfrac{P_h u(t_n)-P_h u(t_{n-1})}{\Delta t}
-P_h \xi(t_n))\\&+\mu(\xi(t_n)-P_h\xi(t_n),\hat{z}^n)
+\rho_1(\Phi_t(t_n)-\dfrac{P_h\Phi(t_n)-P_h\Phi(t_{n-1})}{\Delta t},\hat{e}^n)\\
&+K(\varphi_x(t_n)+\psi(t_n)-((P_h\varphi(t_n))_x
+P_h\psi(t_n)),\hat{e}_{x}^n+\hat{y}^n)\\&-K( e^n_x+y^n,\dfrac{(P_h\varphi(t_n))_x
-(P_h\varphi(t_{n-1}))_x}{\Delta t}-(P_h\Phi(t_n))_x)\\
&-K( e^n_x+y^n,\dfrac{P_h\psi(t_n)-P_h\psi(t_{n-1})}{\Delta t}
-P_h\Psi(t_n))+\gamma(\Phi(t_n)-P_h\Phi(t_n),\hat{e}^n)\\
&+b(\psi_x(t_n)-(P_h \psi(t_n))_x,\hat{y}^{n}_x) -b (y^n_x,
\dfrac{(P_h\psi(t_n))_x-(P_h\psi(t_{n-1}))_x}{\Delta t}
-(P_h\Psi(t_n))_x)\\
&+\rho_3( \vartheta_t(t_n)-\dfrac{P_h \vartheta(t_n)
-P_h \vartheta(t_{n-1})}{\Delta t},\hat{\varrho}^n) 
+\delta(w_x(t_n)-(P_h w(t_n))_x,\hat{\varrho}^n_x)\\
&-\delta(\varrho^n_x,\dfrac{(P_h w(t_n))_x
-(P_h w(t_{n-1}))_x}{\Delta t}-(P_h \vartheta(t_n))_x)
+\kappa(\vartheta_x(t_n)-(P_h \vartheta(t_n))_x,\hat{\varrho}^n_x).
\end{split}
\end{equation*}
Finally, let $Z_n=\rho\rVert\hat{z}^n\lVert^2+\alpha\rVert z^n_x\lVert^2
+\lambda\rVert e^n-z^n\lVert^2+\rho_1\rVert\hat{e}^n\lVert^2
+K\rVert e^n_x+y^n\lVert^2+b\rVert y^n_x\lVert^2+\rho_3\rVert 
\hat{\varrho}^n\lVert^2+\delta\rVert \varrho^n_x\lVert^2$. 
Using Young's inequality, we easily find that
\begin{equation*}
\begin{split}
Z_n-Z_{n-1}\leqslant 2C\Delta t \Bigl( Z_n
&+\left\lVert \xi_t(t_n)-\dfrac{P_h\xi(t_n)-P_h 
\xi(t_{n-1})}{\Delta t} \right\rVert^2
+\lVert u_x(t_n)-(P_hu(t_n))_x \rVert^2\\
&+\left\lVert \dfrac{(P_h u(t_n))_x-(P_h u(t_{n-1}))_x}{\Delta t}
-(P_h \xi(t_n))_x \right\rVert^2\\
&+\lVert \varphi(t_n)-u(t_n)-(P_h\varphi(t_n)-P_h u(t_n)) \rVert^2\\
&+\left\lVert \dfrac{P_h\varphi(t_n)-P_h\varphi(t_{n-1})}{\Delta t}
-P_h\Phi(t_n) \right\rVert^2\\
&+\left\lVert \dfrac{P_h u(t_n)-P_h u(t_{n-1})}{\Delta t}
-P_h \xi(t_n) \right\rVert^2+\left\lVert \xi(t_n)-P_h\xi(t_n) \right\rVert^2\\
&+\left\lVert\Phi_t(t_n)-\dfrac{P_h\Phi(t_n)-P_h\Phi(t_{n-1})}{\Delta t}\right\rVert^2\\
&+ \lVert \varphi_x(t_n)+\psi(t_n)-((P_h\varphi(t_n))_x+P_h\psi(t_n))\rVert^2\\
&+\left\lVert\dfrac{(P_h\varphi(t_n))_x-(P_h\varphi(t_{n-1}))_x}{\Delta t}
-(P_h\Phi(t_n))_x\right\rVert^2\\
&+\left\lVert \dfrac{P_h\psi(t_n)-P_h\psi(t_{n-1})}{\Delta t}-P_h\Psi(t_n)\right\rVert^2
\\&+\lVert \Phi(t_n)-P_h\Phi(t_n) \rVert^2+\lVert \psi_x(t_n)-(P_h \psi(t_n))_x \rVert^2\\
&+\left\lVert \dfrac{(P_h\psi(t_n))_x-(P_h\psi(t_{n-1}))_x}{\Delta t}
-(P_h\Psi(t_n))_x \right\rVert^2\\
& +\left\lVert \vartheta_t(t_n)-\dfrac{P_h \vartheta(t_n)
-P_h \vartheta(t_{n-1})}{\Delta t} \right\rVert^2
+ \lVert w_x(t_n)-(P_h w(t_n))_x \rVert^2\\
&+ \left\lVert  \dfrac{(P_h w(t_n))_x-(P_h w(t_{n-1}))_x}{\Delta t}
-(P_h \vartheta(t_n))_x \right\rVert^2 + \lVert\vartheta_x(t_n)
-(P_h \vartheta(t_n))_x \rVert^2\Bigr).
\end{split}
\end{equation*}
As a consequence, we have
\begin{equation}
\label{step5}	
Z_n-Z_{n-1}\leqslant 2C \Delta t (Z_n+R_n),
\end{equation}
where the residual $R_n$ is the sum of the approximation errors.  
Summing the previous inequality over $n$, it follows that 
\begin{equation*}
Z_n-Z_{0}\leqslant 2C \Delta t \sum_{j=1}^{n} (Z_j+R_j)
\end{equation*}
and, making use of Taylor's expansion in time and (\ref{space:error}) 
to estimate  the time and the space error, we get that
\begin{equation*}
2C \Delta t \sum_{j=1}^{n} R_j\leqslant C(h^2+(\Delta t)^2).
\end{equation*}
Since $Z_{0}=0$, we end up with
\begin{equation*}
Z_n\leqslant 2C \Delta t \sum_{j=1}^{n} Z_j+ C  (h^2+(\Delta t)^2).
\end{equation*}
The result follows by applying a discrete version 
of Gronwall's inequality and taking into account 
that $n \Delta t \leqslant T$.
\end{proof}


\section{Simulations}
\label{sec:06}

In our simulations, we select the following values:
\begin{equation*}
L=1,\;h=0.01,\; \Delta t=h/2,\;\alpha=6,\;\rho_1=2,\;
K=365,\;\rho=\lambda=\mu=\gamma=\beta=b=\rho_3= \delta=\kappa=1,
\end{equation*}
taking as initial conditions
\begin{equation*}
u_0(x)=u_1(x)=\varphi_0(x)= \varphi_1(x)
=\psi_0(x)=w_0(x)=w_1(x)=\sin(\pi x).
\end{equation*} 
The evolution of $u$, $\varphi$,  $\psi$ and $w$  
are represented in 3D in Figures~\ref{fig:1},  
\ref{fig:2}, \ref{fig:3} and \ref{fig:4}, respectively. 
\begin{figure}[b]
\centering
\includegraphics[width=14cm]{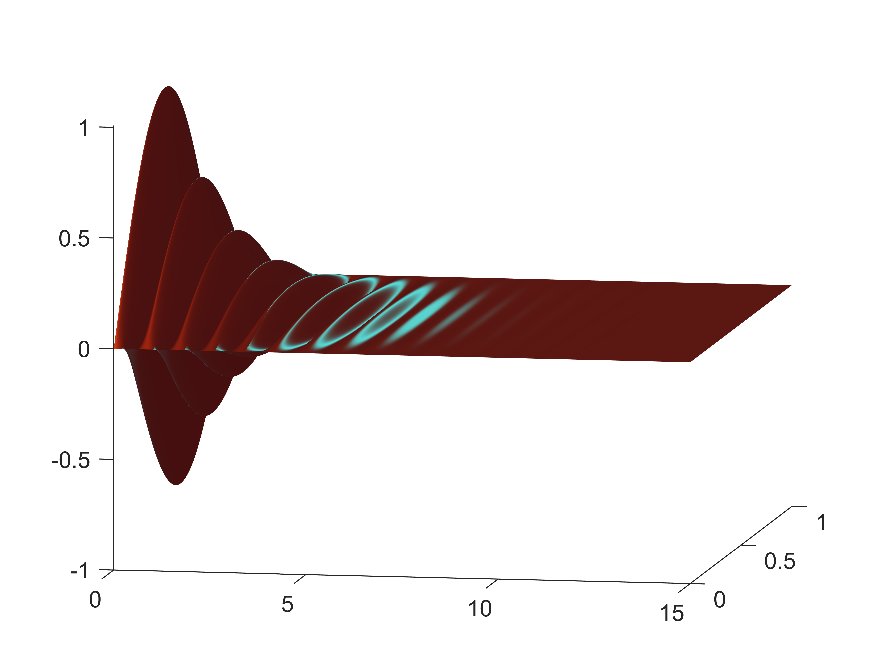} \\
\caption{The evolution in time and space of $u$.}	
\label{fig:1}
\end{figure}
\begin{figure}
\centering
\includegraphics[width=14cm]{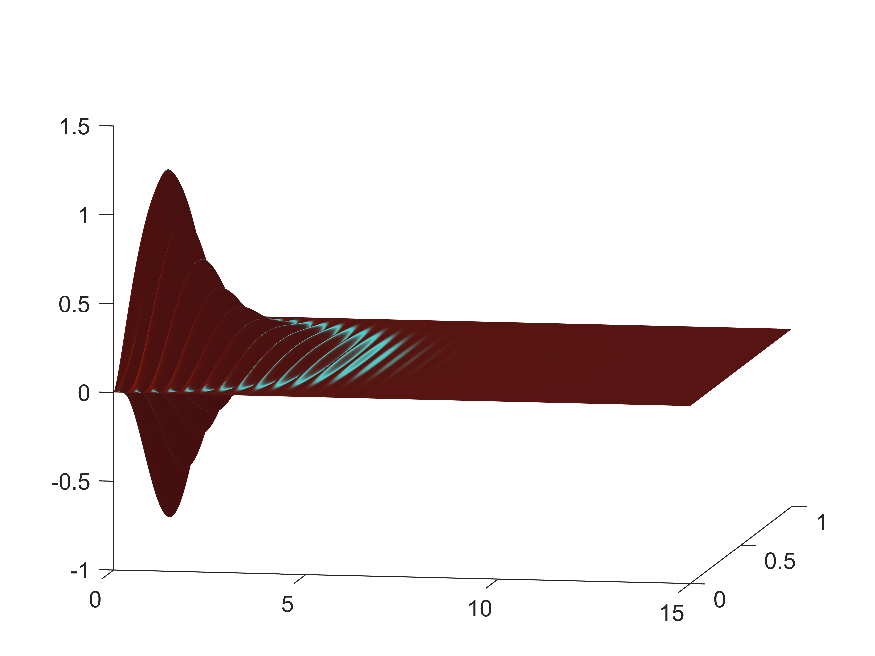} \\
\caption{The evolution in time and space of $\varphi$.}	
\label{fig:2}
\end{figure}	
\begin{figure}
\centering
\includegraphics[width=14cm]{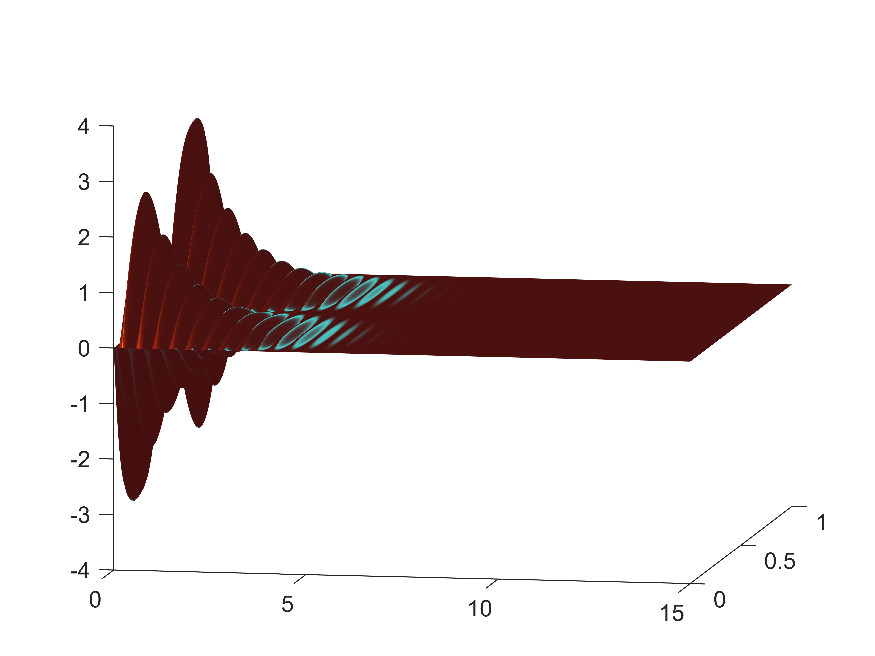} \\
\caption{The evolution in time and space of $\psi$.}
\label{fig:3}	
\end{figure}
\begin{figure}
\centering
\includegraphics[width=14cm]{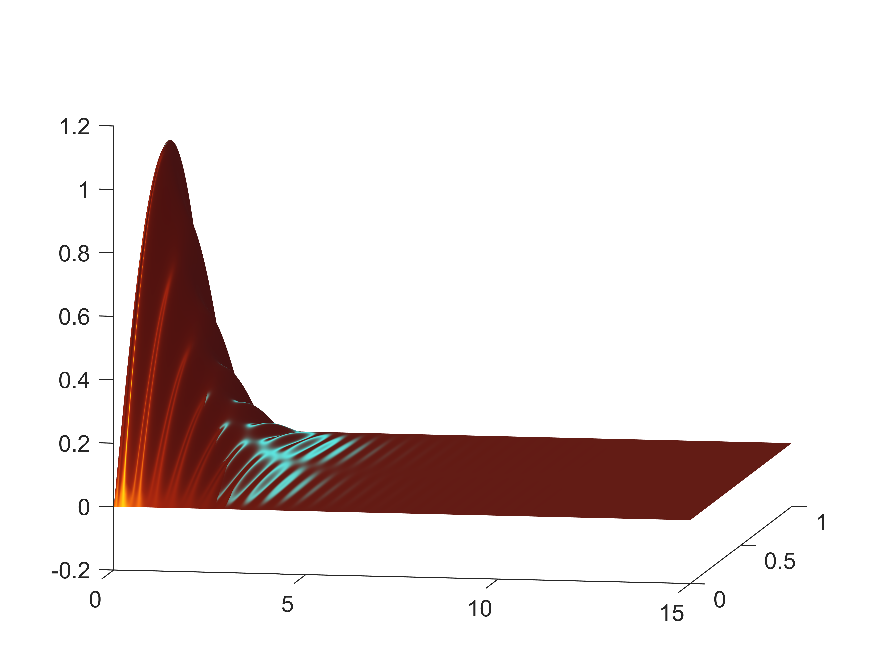} \\
\caption{The evolution in time and space of $w$.}
\label{fig:4}	
\end{figure}

The results at point $x=0.6$ 
are displayed in Figures~\ref{fig:5}, 
\ref{fig:6} and \ref{fig:7}.
\begin{figure}
\centering
\includegraphics[width=14cm]{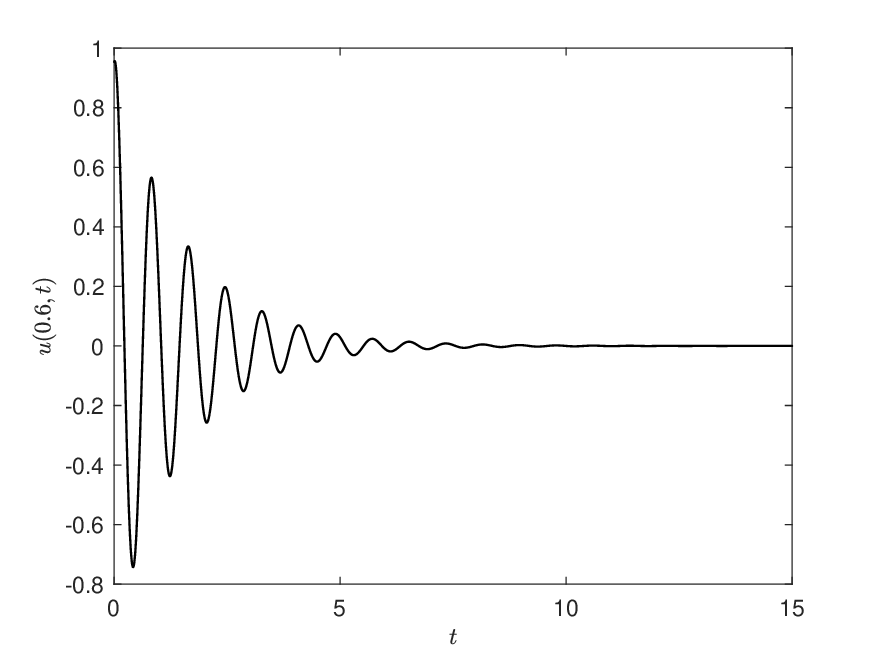} \\
\caption{The evolution in time of $u$ at $x=0.6$.}	
\label{fig:5}
\end{figure}
\begin{figure}
\centering
\includegraphics[width=14cm]{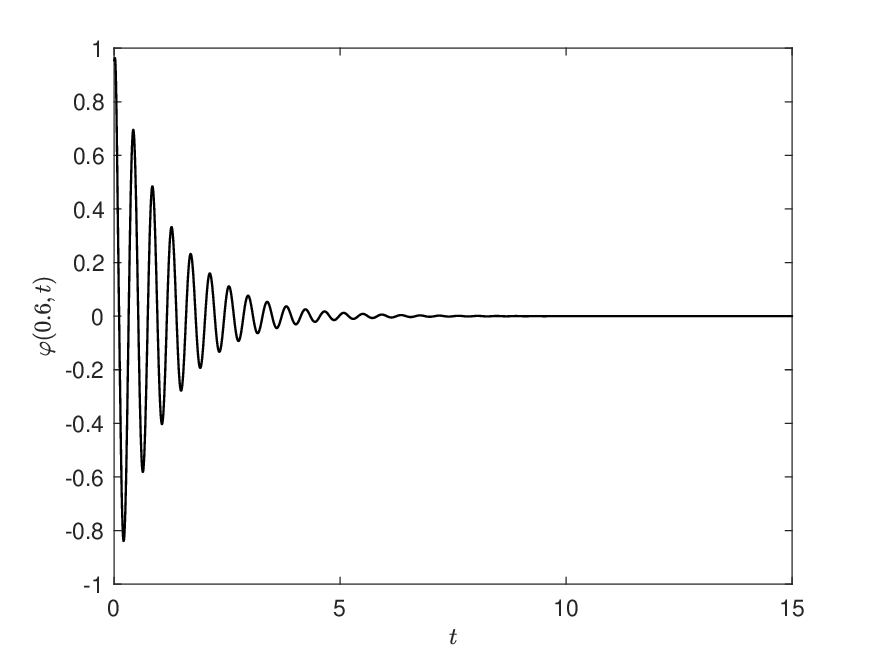} \\
\caption{The evolution in time of $\varphi$ at $x=0.6$.}	
\label{fig:6}
\end{figure}
\begin{figure}
\centering
\includegraphics[width=14cm]{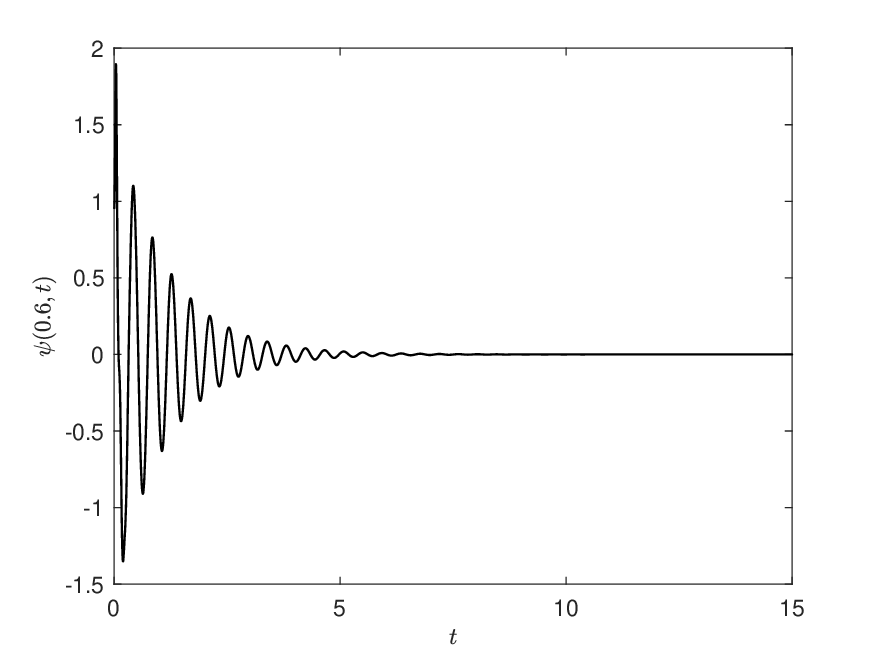} \\
\caption{The evolution in time of $\psi$ at $x=0.6$.}	
\label{fig:7}
\end{figure}

The decay of energy with respect to time 
is shown in Figures~\ref{fig:8}, 
\ref{fig:9} and \ref{fig:10}.
\begin{figure}
\centering
\includegraphics[width=14cm]{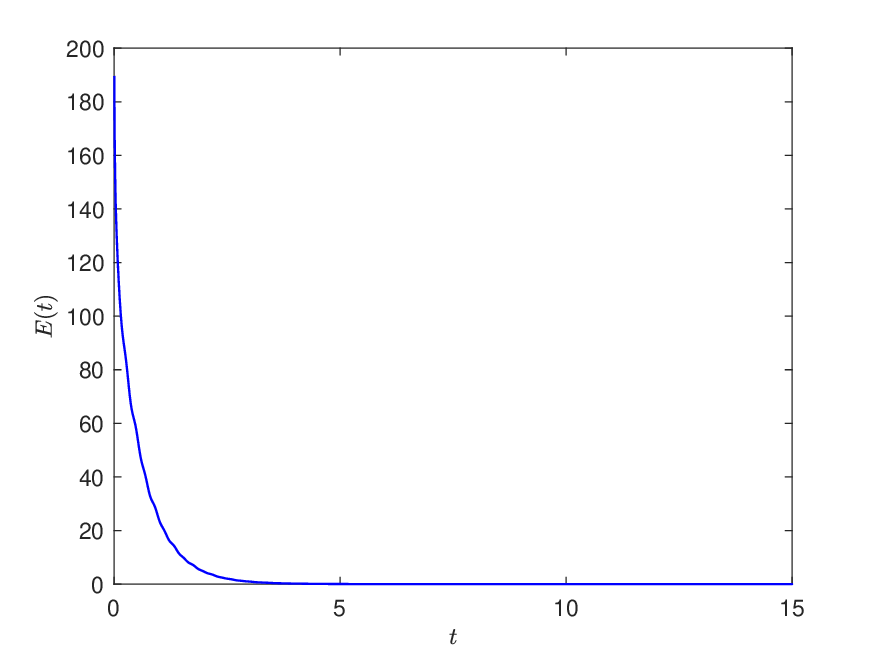} \\
\caption{The evolution in time of energy $E$.}	
\label{fig:8}
\end{figure}
\begin{figure}
\centering
\includegraphics[width=14cm]{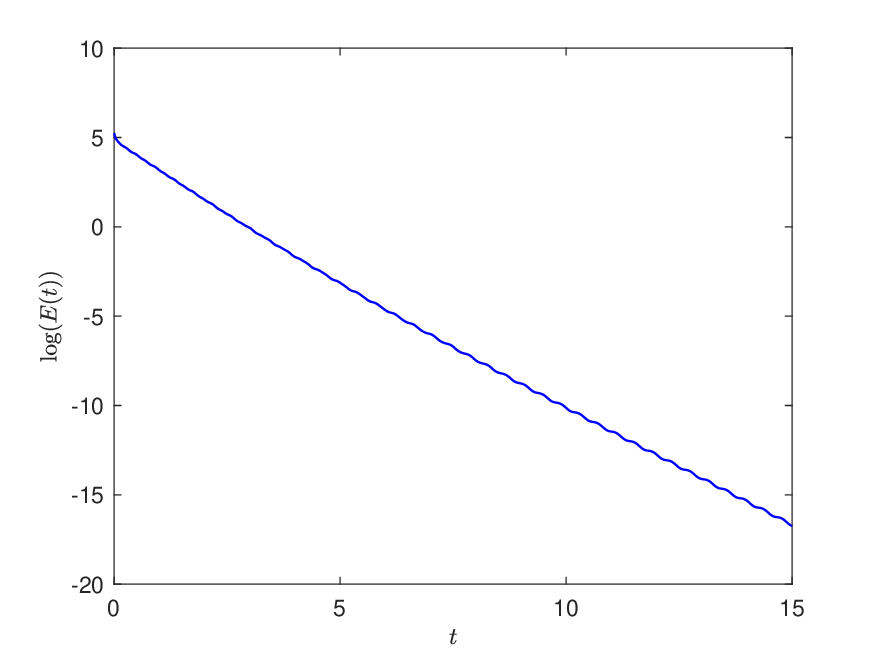} \\
\caption{The evolution in time of $\log(E(t))$.}	
\label{fig:9}
\end{figure}
\begin{figure}
\centering
\includegraphics[width=14cm]{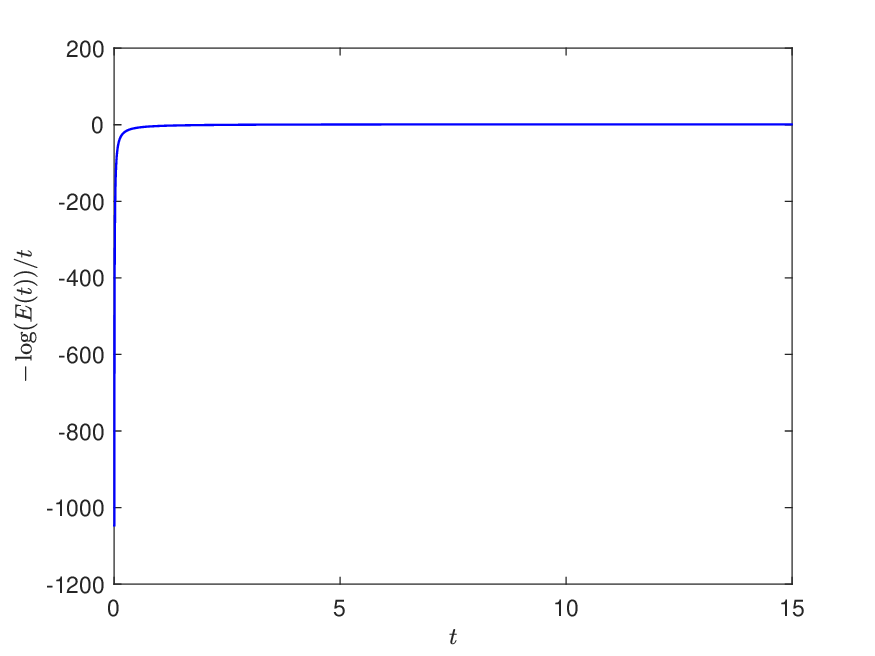} \\
\caption{The evolution in time of $-\log(E(t))/t$.}	
\label{fig:10}
\end{figure}	

Following this, we carried out a numerical simulation 
to evaluate the accuracy of the error estimate. 
We solved the modified problem 
\begin{equation}
\begin{cases}
\rho u _{tt}-\alpha u_{xx} -\lambda\left( \varphi - u\right)
+\mu u_t =f_1,\\ 
\rho _{1}\varphi _{tt} -K\left( \varphi _{x}+\psi \right)_{x}
+\lambda\left( \varphi - u\right)+\gamma \varphi_t +\beta w_{xt}=f_2, \\ 
-b\psi _{xx} +K\left(
\varphi _{x}+\psi \right)=f_3, \\ 
\rho_3 w_{tt}-\delta w_{xx} +\beta\varphi_{xt}-\kappa w_{xxt}=f_4,\\ 
\end{cases}
\end{equation}
where functions $f_1$, $f_2$, $f_3$, $f_4$, 
and the initial data are derived from the exact solution
\begin{equation*}
\begin{split}
u(x,t)&=0.01tx^2(x-1)^2,\\
\varphi(x,t)&=e^t\sin(\pi x),\\
\psi(x,t)&=e^t x\cos(0.5\pi x),\\
w(x,t)&=2e^t\sin(\pi x).\\
\end{split}
\end{equation*}
The calculated errors at time $t = 1.2$ are presented 
in Table~\ref{tab:1}, where the $Error$ is defined as
\begin{equation*}
\begin{split}
Error=&( \rVert \xi^n_h-\xi(t_n)\lVert^2+\rVert u^n_{hx}-(u(t_n))_x\lVert^2
+ \rVert \varphi^n_{h}-u^n_h-(\varphi(t_n)-u(t_n))\lVert^2\\
&+\rVert \Phi^n_h-\Phi(t_n)\lVert^2+\rVert \varphi^n_{hx}
+\psi^n_h-((\varphi(t_n))_x+\psi(t_n))\lVert^2 
+\rVert \psi^n_{hx}-(\psi(t_n))_x\lVert^2\\
&+\rVert \vartheta^n_{h}-\vartheta(t_n)\lVert^2 
+\rVert w^n_{hx}-(w(t_n))_x\lVert^2 )^{\frac{1}{2}}.
\end{split}
\end{equation*}
It can be observed that the errors decrease by a factor of approximately 2 
when the discretization parameters are halved. The linear convergence rate 
is also evident in the curves illustrated in Figure~\ref{fig:11}.
\begin{table}[t]
\caption{Computed errors when $T = 1.2$.}\label{tab:1}	
\begin{center}
\begingroup
\setlength{\tabcolsep}{30pt}
\renewcommand{\arraystretch}{1.5}
\begin{tabular}{c  c c}
\hline
\rule[0cm]{0cm}{0.5cm}	$M$      & $\Delta t$    & $Error$   \\
\hline 
\rule[0cm]{0cm}{0.5cm}
$40$  & $1.00\times10^{-3}$ &
$4.164\times10^{-1}$   \\ 
\rule[0cm]{0cm}{0.5cm}	
$80$    & $5.00\times 10^{-4}$  &
$1.949\times 10^{-1}$   \\ 
\rule[0cm]{0cm}{0.5cm}
$160$  & $2.50\times 10^{-4}$ & 
$9.567\times 10^{-2}$     \\ 
\rule[0cm]{0cm}{0.5cm}
$320$  & $1.25\times 10^{-4}$  &
$4.770\times 10^{-2}$     \\ 
\rule[0cm]{0cm}{0.5cm}
$640$    & $6.25\times 10^{-5}$  & 
$2.402\times 10^{-2}$       \\ 
\rule[0cm]{0cm}{0.5cm}
$1280$  & $3.125\times 10^{-5}$ & 
$1.241\times 10^{-2}$   \\ \hline
\end{tabular}
\endgroup
\end{center}
\end{table}
\begin{figure}
\centering  
\subfigure{\includegraphics[width=14cm]{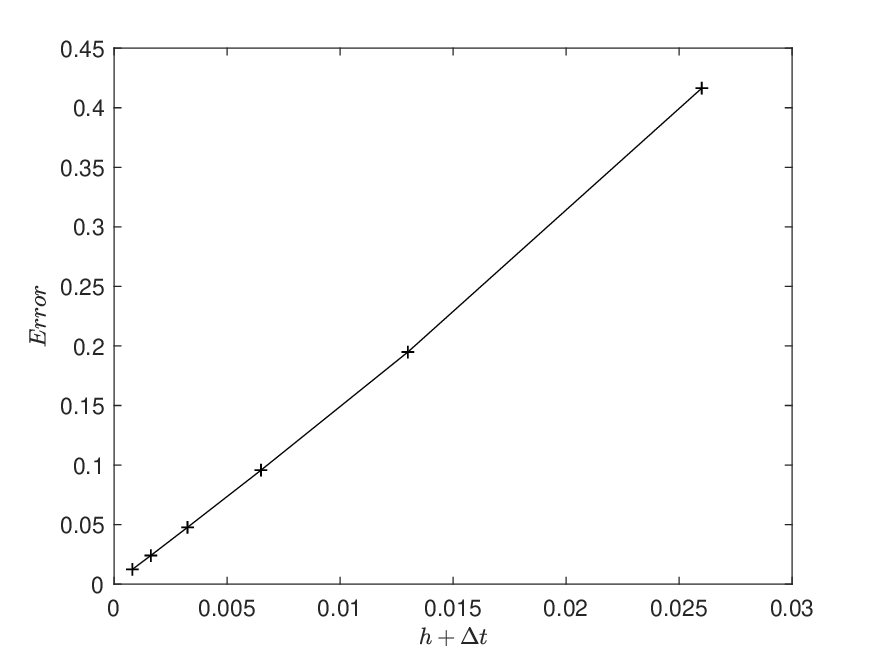}}	
\subfigure{\includegraphics[width=14cm]{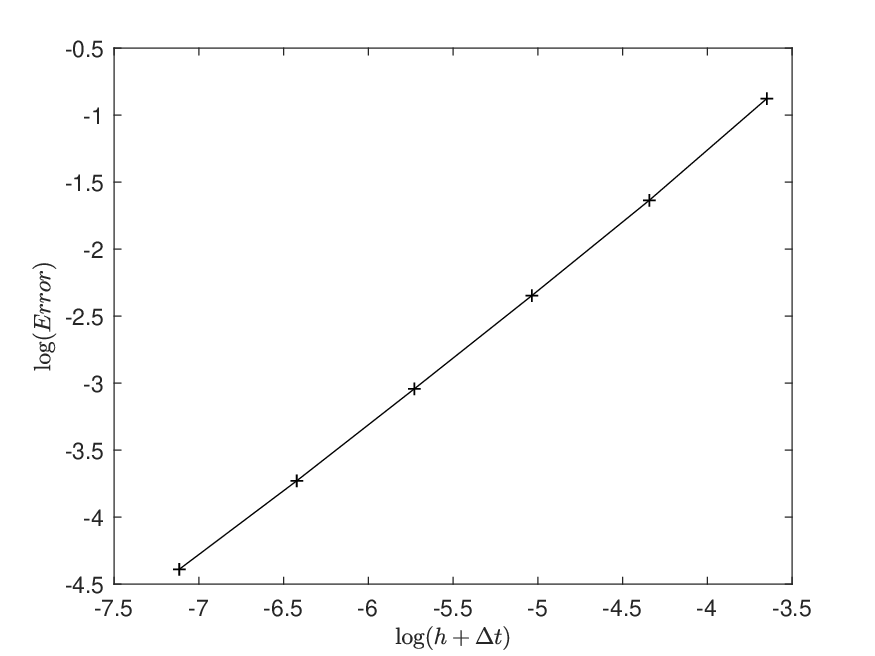}}
\caption{The evolution of the error depending on $h+\Delta t$.}	
\label{fig:11}
\end{figure}


\section*{Acknowledgments}

This research is part of first author's Ph.D., 
which is carried out at Ferhat Abbas S\'etif 1 University.
Chabekh is grateful to the financial support 
of Ferhat Abbas S\'etif 1 University, Algeria,
for a one-month visit to the R\&D Unit CIDMA, 
Department of Mathematics, University of Aveiro. 
The hospitality of the host institution 
is here gratefully acknowledged. 
Torres is supported by the Portuguese 
Foundation for Science and Technology (FCT)
within project UIDB/04106/2020.

The authors are grateful to two
referees for several pertinent questions
and comments that helped them to improve
the originally submitted version.



\bigskip


\end{document}